\documentclass[11pt]{amsart}
\usepackage[margin=1in]{geometry}

\usepackage{amscd}
\usepackage{pb-diagram}
\usepackage[mathscr]{eucal}
\usepackage{amsfonts}
\usepackage{amsmath}
\usepackage{amsthm}
\usepackage{amssymb}
\usepackage{latexsym}
\usepackage{tabularx,array}
\newfont{\msam}{msam10}
\theoremstyle{plain}
        \newtheorem{theorem}{Theorem}[section]
        \newtheorem{lemma}[theorem]{Lemma}
        \newtheorem{remark}[theorem]{Remark}
        \newtheorem{proposition}[theorem]{Proposition}
        \newtheorem{corollary}[theorem]{Corollary}

\theoremstyle{definition}
        \newtheorem{definition}[theorem]{Definition}
\let\nc\newcommand

\nc{\la}{\label}

\def\bthm{\begin{theorem}}
\def\ethm{\end{theorem}}
\def\blemma{\begin{lemma}}
\def\elemma{\end{lemma}}
\def\bproof{\begin{proof}}
\def\eproof{\end{proof}}
\def\bprop{\begin{proposition}}
\def\eprop{\end{proposition}}
\def\bcor{\begin{corollary}}
\def\ecor{\end{corollary}}

\nc{\Hom}{{\rm{Hom}}}
\nc{\Ext}{{\rm{Ext}}}
\nc{\HOM}{\underline{\rm{Hom}}}
\nc{\EXT}{\underline{\rm{Ext}}}
\nc{\TOR}{\underline{\rm{Tor}}}
\nc{\End}{{\rm{End}}}
\nc{\GL}{{\rm{GL}}}
\nc{\PGL}{{\rm{PGL}}}
\nc{\SL}{{\rm{SL}}}
\nc{\PSL}{{\rm{PSL}}}
\nc{\Rep}{{\rm{Rep}}}
\nc{\ad}{{\rm{ad}}}
\nc{\dlim}{\varinjlim}

\newcommand{\tr}{{\rm{tr}}}

\newcommand{\rar}{\rightarrow}

\newcommand{\complex}{{\mathbb C}}

\newcommand{\reals}{{\mathbb R}}

\newcommand{\cala}{{\mathcal A}}
\newcommand{\calb}{{\mathcal B}}
\newcommand{\calc}{{\mathcal C}}
\newcommand{\cald}{{\mathcal D}}
\newcommand{\cale}{{\mathcal E}}
\newcommand{\calf}{{\mathcal F}}

\newcommand{\calh}{{\mathcal H}}

\newcommand{\calk}{{\mathcal K}}
\newcommand{\call}{{\mathcal L}}
\newcommand{\calm}{{\mathcal M}}

\newcommand{\calo}{{\mathcal O}}

\newcommand{\calr}{{\mathcal R}}

\newcommand{\calv}{{\mathcal V}}

\newcommand{\calw}{{\mathcal W}}

\begin{document}

\title{Hochschild Lefschetz Class for $\cald$-modules}
\author{Ajay Ramadoss}
\address{Departement Mathematik,
 ETH Z\"{u}rich, R\"{a}mistrasse 101, 8092 Z\"{u}rich}
\email{ajay.ramadoss@math.ethz.ch}
\author{Xiang Tang}
\address{Department of Mathematics, Washington University, St.
Louis, MO 63130, USA} \email{xtang@math.wustl.edu}
\author{Hsian-hua Tseng}
\address{Department of Mathematics, Ohio State University, Columbus, OH 43210, USA} \email{hhtseng@math.ohio-state.edu}
\maketitle
\begin{abstract}
We introduce a notion of Hochschild Lefschetz class for a good coherent $\cald$-module on a compact complex manifold, and prove that this class is compatible with the direct image functor. We prove an orbifold Riemann-Roch formula for a $\cald$-module on a compact complex orbifold. 
\end{abstract}

\section{Introduction}

 In \cite{KS}, Kashiwara and Schapira systematically studied the Hochschild class for deformation
 quantization algebroids. As an application, they obtained a new way to define the Euler class of a good 
coherent $\cald$-module on a complex manifold, introduced by Schapira and Schneiders \cite{ss2}. 
In this paper, we aim to generalize the notion of Hochschild class to the equivariant setting. 

Let $M$ be a compact complex manifold and $\cald_M$  the sheaf of holomorphic differential operators on $M$.
  A coherent $\cald_M$-module $\calm$ is called ``good" if for any compact subset of $M$ there is a neighborhood in which $\calm$ admits a finite filtration $(\calm_k)$ by coherent $\cald_M$-submodules such that each quotient $\calm_k/\calm_{k-1}$ can be endowed with a good filtration. 
 We denote by $D^b(\cald_M)$ the bounded derived category of $\cald_M$-modules, and
 $D^b_{\text{coh}}(\cald_M)$ the full triangulated subcategory of $D^b(\cald_M)$ consisting of objects with 
coherent cohomologies. Let $X:=T^*M$ be the cotangent bundle of $M$. Following \cite{KS}, we consider the 
sheaf $\widehat{\cale}_X$ of formal microdifferential operators  on $X$. Let $\pi_M: X=T^*M\to M$ be the canonical projection. 
There is a natural flat embedding map  $\pi^{-1}_M \cald_M\hookrightarrow \widehat{\cale}_X$. 
This gives a natural functor from $D^b_{\rm coh}(\cald_M)$ to $D^b_{\rm coh}(\widehat{\cale}_X)$. 
Such a functor allows the use of microlocal techniques to study $\cald_M$-modules. 
Let $\calh\calh(\widehat{\cale}_X, \widehat{\cale}_X)$ be the $\complex$-sheaf\footnote{This is actually a slight abuse of terminology. In fact, $\calh\calh(\widehat{\cale}_X, \widehat{\cale}_X)$ is an object in the derived category of sheaves of $\complex$-vector spaces on $X$.} of Hochschild homologies of
 $\widehat{\cale}_X$ on $X$. For $\calm\in D^b_{\rm coh}(\cald_M)$ and an element $u\in Hom_{\cald_M}(\calm, \calm)$, Kashiwara and Schapira \cite{KS} introduced a {\em Hochschild class} $$hh(\calm, u)\in {H}^0_{{\rm supp}(\calm)}(X;\calh\calh(\widehat{\cale}_X, \widehat{\cale}_X)).$$  
 The image of the the Hochschild class $hh(\calm,u)$ under the 
quasi-isomorphism $$\calh\calh(\widehat{\cale}_X, \widehat{\cale}_X)\to \complex[\dim_X]$$ is called the 
{\em Euler class} of $(\calm, u)$. As Hochschild homology behaves well under direct images, the Hochschild class $hh(\calm, u)$ satisfies a nice formula \cite[Theorem 4.3.5]{KS} under the direct image functor. This formula is analogous to the direct image  property of the Euler class of $(\calm, u)$, which was proved by Schapira and Schneiders \cite{ss2}. 

In this paper we consider a holomorphic diffeomorphism $\gamma$ on $M$.  Let $\calm$ be a $\cald_M$-module. Then, the sheaf 
$\gamma_*\calm$ of $\complex$-vector spaces on $M$ has a natural $\cald_M$-module structure. 
This in turn gives a natural functor $\gamma_*: D^b_{\rm coh}(\cald_M)\to D^b_{\rm coh}(\cald_M)$.  
A similar construction and functor can be introduced for $\widehat{\cale}_X$-modules. Given an element $u\in Hom_{\cald_M}(\calm, \gamma_*(\calm))$, we introduce a  {\em Hochschild Lefschetz class}\footnote{See Eq. (\ref{eq:hoch-hom}) for the definition of $\calh\calh(\widehat{\cale}_X, \widehat{\cale}_X^\gamma)$.} $$hh^\gamma(\calm, u)\in H^0(X, \calh\calh(\widehat{\cale}_X, \widehat{\cale}_X^\gamma))$$ for $X=T^*M$. Our construction coincides with the Kashiwara-Schapira Hochschild class $hh(\calm, u)$ when $\gamma=id$. The {\em Lefschetz class} of $u$ is defined to be the image of $hh^\gamma(\calm, u)$ under the quasi-isomorphism
 $$\calh\calh(\widehat{\cale}_X, \widehat{\cale}_X^\gamma)\to \iota_!(\complex_{X^\gamma}[\dim(X^\gamma)]).$$ Like the Hochschild class, we prove that the Hochschild Lefschetz class satisfies nice formulas under the direct image functor. 
We expect that this approach will provide a relatively easy route to results about the Lefschetz class introduced
 by Guillermou \cite{G}. Let $\Gamma$ be a finite group acting on $M$ by holomorphic diffeomorphisms. We apply 
our developments to study the Hochschild class and the Euler class of a good 
$\Gamma$-equivariant coherent $\cald_M$-module $\calm$.
 In this situation, every $\gamma\in \Gamma$ naturally defines an element 
$\gamma$ in $Hom_{\cald_M}(\calm, \gamma_*(\calm))$.  We can use the expression 
$$
\frac{1}{|\Gamma|} \sum_{\gamma\in \Gamma} hh^\gamma(\calm, \gamma)
$$ 
to introduce the orbifold Hochschild class of $\calm$ on the quotient orbifold $Q_X:=X/\Gamma=T^*M/\Gamma$. 
We prove that this description of the orbifold Hochschild class for $\calm$ is 
equivalent to the more abstract 
definition that arises by working with the sheaf of algebras
 $\cald_M\rtimes \Gamma$ (and $\widehat{\cale}_X\rtimes \Gamma$) 
over $Q_M:=M/\Gamma$ (and $Q_X$) using 
techniques developed by Bressler, Nest, and Tsygan \cite{BNT}.

The main result of this paper is a Riemann-Roch formula for the Euler class of a good
$\Gamma$-equivariant  coherent $\cald_M$-module $\calm$ on $M$. We prove that (see Theorem \ref{thm:obfld-ss})
\[
{\rm eu}_{Q}(\calm)=\Big(\frac{1}{m}{\rm ch}_{Q}(\sigma_{\text{char}(\calm)}(\calm))\wedge {\rm eu}_{Q}(N)\wedge \pi_M^*Td(IQ_M)\Big)_{\dim(I Q_X)}.
\]
Hereby, $IQ_X$ (and $IQ_M$) is the inertia orbifold associated to the orbifold $Q_X$ (and $Q_M$); 
${\rm ch}_{Q}$ is the orbifold Chern character for the orbifold K-theory element $\sigma_{ch(\calm)}(\calm)$; ${\rm eu}_{Q}(N)$ is a characteristic class associated to the normal bundle $N$ of the local embedding
 $IQ\rightarrow Q$; $Td(IQ_M)$ is the Todd class of the orbifold bundle $TIQ_M$ over $IQ_M$; $\pi_M$ is the canonical projection from $IQ_X$ to $IQ_M$;  
and $m$ is the locally constant function on $IQ_X$ measuring the size of the isotropy group at each point. The proof generalizes the original idea of Bressler, Nest, and Tsygan \cite{BNT} along the developments in \cite{PPT} and \cite{PPT2}.

The paper is organized as follows. We start with fixing some basic notations in  Sec. \ref{sec:notations}. In Sec. \ref{sec:lefschetz}, we introduce the construction of the Hochschild Lefschetz class and orbifold Euler class for a good coherent
 $\cald_M$-module, and discuss their properties.  In Sec. \ref{sec:euler},
 we explain the computation of the orbifold Euler (Chern) class. \\

\noindent{\bf Acknowledgments:} We would like to thank Pierre Schapira and Giovanni Felder for interesting discussions. A.R. is supported by the Swiss National Science Foundation for the project ``Topological quantum mechanics and index theorems" (Ambizione Beitrag Nr.$\text{ PZ}00\text{P}2\_127427/1$). X.T. is partially supported by NSF grant DMS-0900985.

\section{Basic Notations}\label{sec:notations}
Throughout this paper, we closely follow the terminologies and conventions introduced in \cite{KS}.  Let $M$ be a complex manifold of complex dimension $d_M$. In this paper, dimension of a complex manifold/orbifold always refers to its complex dimension. Consider  $\cald_M$ the sheaf of holomorphic differential operators on $M$ and $\cald_{M^a}$ the sheaf of holomorphic differential operators on $M$ with the opposite algebra structure from $\cald_M$. A coherent $\cald_M$-module $\calm$ is called ``good" if in a neighborhood of any compact subset of $M$, $\calm$ admits a finite filtration $(\calm_k)$ by coherent $\cald_M$-submodules such that each quotient $\calm_k/\calm_{k-1}$ can be endowed with a good filtration (see \cite[Definition 4.24]{K}).  We denote by $D^b(\cald_M)$ the bounded derived category of $\cald_M$ modules, and $D^b_{\text{coh}}(\cald_M)$ the full triangulated subcategory of $D^b(\cald_M)$ consisting of objects with coherent cohomologies.  

Let $X=T^*M$ be the cotangent bundle with dimension $d_X=2d_M$ with the projection $\pi_M:X=T^*M\rightarrow M$. Denote by $\Omega^i_X$ the sheaf of holomorphic $i$-forms on $X$. On $X=T^*M$, consider the filtered sheaf  $\widehat{\cale}_{X}$ of $\complex$-algebras of formal microdifferential operators, and the subsheaf $\widehat{\cale}(0)_{X}$ of operators of order $\leq0$. Let $\pi^{-1}_M \cald_M$ be the pullback of $\cald_M$ on $X$.  Denote by $D^b(\widehat{\cale}_X)$ (and $D^b_{\text{coh}}(\widehat{\cale}_{X})$) the bound derived category of $\widehat{\cale}_X$-modules (and the subcategory of objects with coherent cohomologies.) There is a natural morphism $\pi^{-1}_M \cald_M\hookrightarrow \widehat{\cale}_{X}$ such that $\widehat{\cale}_X$ is flat over $\pi^{-1}_M \cald_M$. Given a coherent $\cald_M$-module $\calm$, $$\widehat{\calm}:=\widehat{\cale}_X\otimes _{\pi^{-1}_M \cald_M} \pi_M^{-1}\calm$$ defines a coherent $\widehat{\cale}_X$-module.  In this paper, we will mainly work with the $\widehat{\cale}_X$-module $\widehat{\calm}$ associated to $\calm$.  The support of $\widehat{\calm}$ is called the {\em characteristic variety} of $\calm$ and denoted by $\text{char}(\calm)$. Denote by $\widehat{\cale}_{X^a}$ the sheaf of formal microdifferential operators on $X=T^*M$ with the opposite algebra structure from $\widehat{\cale}_X$. 

Define $\omega:=\Omega^{d_X}_X[d_X]$, and define the duality functor $D'_{\widehat{\cale}_X}$ by
\[
D'_{\widehat{\cale}_X}(\calm):=R\calh om_{\widehat{\cale}_X}(\calm, \widehat{\cale}_X)\in D^b(\widehat{\cale}_{X^a}), \quad \text{for an }\widehat{\cale}_X \text{-module } \calm.
\] 

Let $\Gamma$ be a finite group acting on $M$ holomorphically and also on $X=T^*M$.
Note that for any $\gamma \in \Gamma$, one has a natural isomorphism 
$\gamma^{-1}\widehat{\cale}_X \rar \widehat{\cale}_X$ of sheaves of $\complex$-algebras 
on $X$. Hence, for  any $\gamma\in \Gamma$, any $\widehat{\cale}_X$-module $\calm$ has the natural structure of a $\gamma^{-1}\widehat{\cale}_X$-module. 
Consequently, the pushforward $\gamma_*\calm$ (in the category of sheaves of $\complex$-vector spaces on $X$) has the natural structure of a $\widehat{\cale}_X$-module.
 It is easy to verify that the (right derived functor of) $\gamma_*$ gives rise to a 
functor $\gamma_*$ from $D^b(\widehat{\cale}_X)$ (resp., $D^b_{\text{coh}}(\widehat{\cale}_X)$) to itself.

We denote by $\delta:X\rightarrow X\times X$ the diagonal embedding. For $\gamma\in \Gamma$, let 
$$\delta^\gamma_X: X\rightarrow X\times X,\ \delta^\gamma_X(x):=(\gamma(x),x)$$ 
be the graph of the action of $\gamma$.
  Let $\calc_X$  be the $\widehat{\cale}_{X\times X^a}$-module 
$\delta_{X,*}\widehat{\cale}_X$, and let $\calc_X^\gamma$ be the $\widehat{\cale}_{X\times 
X^a}$-module $\delta^\gamma_{X,*} \widehat{\cale}_X$.  
The sheaf of Hochschild homologies\footnote{This is a minor abuse of terminology.} $\calh\calh(\widehat{\cale}_X)$ (resp., 
 $\gamma$-twisted Hochschild homologies 
$\calh\calh(\widehat{\cale}_X, \widehat{\cale}^\gamma_X)$) is defined to be 
the object
\begin{equation}\label{eq:hoch-hom}
\calh\calh_X(\widehat{\cale}_X, \widehat{\cale}_X):=\delta_{X}^{-1}(\calc_{X^a}\stackrel{L}{\otimes}_{\widehat{\cale}_{X\times X^a}}\calc_X)
\quad (\text{resp.,} \quad \calh\calh(\widehat{\cale}_X, \widehat{\cale}_X^\gamma)
:={\delta_{X}}^{-1}(\calc_{X^a}\stackrel{L}{\otimes}_{\widehat{\cale}_{X\times X^a}}\calc^{\gamma}_X)\,\,\,)
\end{equation}
of the derived category of sheaves of $\complex$-vector spaces on $X$. 
 For any object $\calf$ in the derived category of sheaves of $\complex$-vector spaces on $X$, $H^{\bullet}(X;\calf)$
 shall denote the {\it hypercohomology} of $X$ with coefficients in $\calf$. 

The completed tensor products $\underline{\otimes}$, $\underline{\boxtimes}$, etc. have exactly the same meaning as in~\cite{KS}.
 
\section{Lefschetz class}\label{sec:lefschetz}
\subsection{Definition of Lefschetz class}
Let $\calm$ be a good coherent $\cald_M$ module on $M$. We apply the functor $$\calm\mapsto \widehat{\calm}:=\widehat{\cale}_{X}\otimes_{\pi^{-1}_M \cald_M}\pi_M^{-1}\calm$$ and work with the corresponding $\widehat{\cale}_X$-module $\widehat{\calm}$. 
Let $\gamma$ act on $M$ holomorphically.  Lift the action of $\gamma$ to an action on $X:=T^*M$.  
An element $u$ in $Hom_{\cald_M}(\calm, \gamma_*(\calm))$ naturally defines an element $\hat{u}$ in 
$Hom_{\widehat{\cale}_X}(\widehat{\calm}, \gamma_*(\widehat{\calm}))$.  
In what follows, we will introduce a Lefschetz class for $u$ by studying $\hat{u}$ in $Hom_{\widehat{\cale}_X}(\widehat{\calm}, \gamma_*(\widehat{\calm}))$. 
Our construction generalizes analogous constructions in \cite{KS}. 

\begin{lemma}\label{lemma:mor}
Let $\widehat{\calm} \in D^b_{\rm coh}(\widehat{\cale}_X)$. There is a natural morphism in $D^b_{\rm coh}(\widehat{\cale}_{X\times X^a})$:
\begin{eqnarray}
\label{eq:cotrace} 
\label{eq:trace}&\gamma_*(\widehat{\calm}) \stackrel{L}{\underline{\boxtimes}}D'_{\widehat{\cale}_X}(\widehat{\calm}) \rightarrow \calc_X^\gamma.
\end{eqnarray}
\end{lemma}
\begin{proof}
By \cite[Lemma 4.1.1]{KS}, there is a natural morphism 
\[
\widehat{\calm} \stackrel{L}{\underline{\boxtimes}}D'_{\widehat{\cale}_X}(\widehat{\calm}) \rightarrow \calc_X.
\]
Applying the functor $(\gamma \times 1)_*$ to the above morphism, we obtain the desired morphism 
\[
\gamma_*(\widehat{\calm}) \stackrel{L}{\underline{\boxtimes}}D'_{\widehat{\cale}_X}(\widehat{\calm}) \rightarrow \calc_X^\gamma.
\]
\end{proof}

Let $u\in Hom_{\widehat{\cale}_X}(\widehat{\calm}, \gamma_*(\widehat{\calm}))$. Consider the morphism
\[
\begin{split}
R\calh om_{\widehat{\cale}_X}(\widehat{\calm}, \gamma_*(\widehat{\calm}) )&\stackrel{\sim}{\leftarrow}D'_{\widehat{\cale}_X}(\widehat{\calm})\stackrel{L}{\otimes}_{\widehat{\cale}_X}\gamma_*(\widehat{\calm})\\
&\cong \calc _{X^a}\stackrel{L}{\otimes}_{\widehat{\cale}_{X\times X^a}}\big(\gamma_*(\widehat{\calm})\stackrel{L}{\underline{\boxtimes}}D'_{\widehat{\cale}}(\widehat{\calm})\big)\\
&\rightarrow \calc _{X^a}\stackrel{L}{\otimes}_{\widehat{\cale}_{X\times X^a}}\calc_{X}^\gamma\xrightarrow{\delta^{-1}_X}\calh\calh(\widehat{\cale}_X, \widehat{\cale}_X^\gamma).
\end{split}
\]
This defines a natural map
\begin{equation} \label{as3e1}
Hom_{\widehat{\cale}_X}(\widehat{\calm}, \gamma_*(\widehat{\calm}))\longrightarrow  H^0_{\text{supp}(\calm)}(X;\calh\calh(\widehat{\cale}_X, \widehat{\cale}^\gamma_X)).
\end{equation}
\begin{definition}\label{def:lefschetz}
For an element $u\in Hom_{\cald_M}(\calm, \gamma_*(\calm))$, define the Hochschild Lefschetz class $hh^\gamma(\calm, u)\in H^0_{\text{supp}(\calm)}(X;\calh\calh(\widehat{\cale}_X, \widehat{\cale}_X^\gamma))$ to be the image of $\hat{u}\in Hom_{\widehat{\cale}_X}(\widehat{\calm}, \gamma_*(\widehat{\calm}))$ under the morphism~\eqref{as3e1}.
\end{definition}

Recall that $\calh\calh(\widehat{\cale}_X,\widehat{\cale}_X^{\gamma})$ can be naturally identified with 
$RHom_{\widehat{\cale}_{X \times X^a}}(\omega_X^{\otimes-1}, \calc_X^{\gamma})$ using the duality functor $D'_{\widehat{\cale}_{X\times X^a}}$ and \cite[Theorem 2.5.7]{KS}. The following analog of~\cite[Lemma 4.1.4]{KS} holds. We leave its proof to the interested reader. 

\begin{lemma} \label{lks4.1.4}
Under the natural identification of $\calh\calh(\widehat{\cale}_X,\widehat{\cale}_X^{\gamma})$ with 
$RHom_{\widehat{\cale}_{X \times X^a}}(\omega_X^{\otimes-1}, \calc_X^{\gamma})$,  the Hochschild Lefschetz class 
$hh^\gamma(\calm, u)$ coincides with the following composite of morphisms 
\[ \omega_X^{\otimes-1} \xrightarrow{ \text{\cite[Lemma 4.1.1 (i)]{KS}}  }
\widehat{\calm} \stackrel{L}{\underline{\boxtimes}} D'_{\widehat{\cale}_X}(\widehat{\calm}) 
\xrightarrow{\hat{u} \boxtimes \text{id}}\gamma_*(\widehat{\calm}) \stackrel{L}{\underline{\boxtimes}} D'_{\widehat{\cale}_X}(\widehat{\calm}) 
\xrightarrow{\text{Lemma \ref{lemma:mor}}}\calc_X^{\gamma} \,\text{.}
\]
\end{lemma}

Let $X^\gamma$ be the submanifold\footnote{$X^\gamma$ is a disjoint union of embedded submanifolds possibly of different dimensions.} of $X$ consisting of $\gamma$-fixed points and let $\iota:X^\gamma \hookrightarrow X$ be the inclusion. Let $\Omega^\bullet_{X^\gamma}$ be the (smooth) de Rham complex on $X^\gamma$ (viewed 
as a complex of sheaves on $X^{\gamma}$). As in \cite[Section 3]{PPT2}, one can construct a (distinguished) quasi-isomorphism  with the quasi-isomorphism $$\calh\calh(\widehat{\cale}_X, \widehat{\cale}_X^\gamma) \rightarrow \Omega^{(\text{dim}(X^{\gamma}))-\bullet}$$ following a construction in \cite[Section 4]{FFS} and \cite[Section 2]{EF} (see also~\cite{FT}). Therefore, 
\begin{proposition}\label{prop:gamma-hoch} The Hochschild homology $\calh\calh(\widehat{\cale}_X, \widehat{\cale}_X^\gamma)$  is quasi-isomorphic (via a distinguished quasi-isomorphism) to $\iota_!(\complex_{X^\gamma}[\dim(X^\gamma)])$.
 \end{proposition}

Thus one can define the {\em (microlocal) Lefschetz class} $ \mu \text{eu}^\gamma(\calm, u)\in H^{\dim{X^\gamma}}_{\text{supp}(\calm)^\gamma}(X^\gamma, \complex)$ of $u$ to be the image of $hh^\gamma(\calm, u)$ under the (distinguished) quasi-isomorphism 
$$
\calh\calh(\widehat{\cale}_X, \widehat{\cale}_X^\gamma)\rightarrow 
\iota_!(\complex_{X^\gamma}[\dim(X^\gamma)]).
$$ 
When $M$ is a point, $\gamma$ acts on $M$ trivially. Then $ \mu \text{eu}(\calm, u)$ is equal to the trace of $u$ as an endomorphism of $\calm$. 

\begin{remark}
One also has a class $\text{eu}(\calm, u) \in H^{2\dim{M^{\gamma}}}(M^{\gamma}, \complex)$. Its construction is completely analogous to that of $ \mu \text{eu}^\gamma(\calm, u)$. When $\gamma=\text{Id}$ and when $\calm = \cald_M \otimes_{\mathcal O_M} \cale$ for some holomorphic vector bundle $\cale$ on $M$, then $ \text{eu}(\calm, u)$ is equal to the trace of $u$ as an endomorphism of $\mathcal O_M \otimes^L_{\cald_M}\calm$ (see~\cite{EF,FT,R}).   
\end{remark}

\subsection{Composition of Hochschild Lefschetz classes}

Consider three complex manifolds $M_i,\ i=1,2,3$, and $X_i=T^*M_i,\ i=1,2,3$. Let $\widehat{\cale}_{X_1\times X_2^a}$ and $\widehat{\cale}_{X_2\times X_3^a}$ be the sheaf of formal microdifferential operators on $X_i\times X_{i+1}^a,\ i=1,2$.  Assume that the group $\Gamma$ acts on $M_i$ and $X_i$ holomorphically. Let $p_{ij}$ be the canonical projection from $X_1\times X_2\times X_3$ to $X_i\times X_j$ for $1\leq i<j\leq 3$.
 Also, let $d_i,\ i=1,2,3$ denote the complex dimensions of the $X_i$. 
In this subsection, as in~\cite{KS}, 
we implicitly identify $X=T^*M$ with its image in $X \times X$ under the embedding 
$\delta^{\gamma}_X$ whenever required. Also\footnote{As did in~\cite{KS}.}, to simplify notations, 
we sometimes denote $\widehat{\cale}_{X_i}$ by $\widehat{\cale}_i$: for 
example, $\otimes_{22^a}$ actually stands for $\otimes_{\widehat{\cale}_{X_2 \times X_2^a}}$.  

\begin{proposition}\label{prop:morphism}There is a natural morphism 
\[
\circ: Rp_{13!}\big(p^{-1}_{12}\calh\calh(\widehat{\cale}_{X_1\times X_2^a}, \widehat{\cale}_{X_1\times X_2^a}^\gamma)\stackrel{L}{\otimes} p^{-1}_{23}\calh\calh(\widehat{\cale}_{X_2\times X_3^a}, \widehat{\cale}_{X_2\times X_3^a}^{\gamma})\big)\longrightarrow \calh\calh(\widehat{\cale}_{X_1\times X_3^a}, \widehat{\cale}_{X_1\times X_3^a}^{\gamma}).
\]
\end{proposition}
\begin{proof} Following \cite{KS}, we will denote by $\widehat{\cale}_{Z_i}$ the complex manifold $\widehat{\cale}_{X_i\times X_{i}^a}$, and identify the Hochschild homology $\calh\calh(\widehat{\cale}_{X_i\times X_j^a}, \widehat{\cale}^\gamma_{X_i\times X_j^a})$ as follows:
\begin{eqnarray*}
\calh\calh(\widehat{\cale}_{X_i\times X_j^a}, \widehat{\cale}^\gamma_{X_i\times X_j^a}) &&\cong\big(\calc_{X_i^a}\stackrel{L}{\underline{\boxtimes}}\calc_{X_j}\big)\stackrel{L}{\otimes }_{\widehat{\cale}_{Z_i\times Z_j^a}}\big(\calc^\gamma_{X_i}\stackrel{L}{\underline{\boxtimes}}\calc_{X_j^a}^\gamma\big)\\
&&\cong RHom_{\widehat{\cale}_{Z_i\times Z_j^a}}\big(\omega_{X_i}^{\otimes -1}\stackrel{L}{\underline{\boxtimes}}\omega_{X^a_{j}}^{\otimes -1}, \calc^\gamma_{X_i}\stackrel{L}{\underline{\boxtimes}}\calc^\gamma_{X_j^a}\big)\\
&&\cong RHom_{\widehat{\cale}_{Z_i\times Z_j^a}}\big((\omega_{X_i}^{\otimes -1}\stackrel{L}{\underline{\boxtimes}}\omega_{X^a_{j}}^{\otimes -1})\stackrel{L}{\underline{\otimes}}_{\widehat{\cale}_{X^a_j}}\omega_{X_j^a}, (\calc^\gamma_{X_i}\stackrel{L}{\underline{\boxtimes}}\calc^\gamma_{X_j^a})\stackrel{L}{\underline{\otimes}}_{\widehat{\cale}_{X_j^a}}\omega_{X^a_j}\big)\\
&&\cong RHom_{\widehat{\cale}_{Z_i\times Z_j^a}}\big(\omega_{X_i}^{\otimes -1}\stackrel{L}{\underline{\boxtimes}}\calc_{X^a_{j}}, \calc^\gamma_{X_i}\stackrel{L}{\underline{\boxtimes}}\omega^\gamma_{X_j^a}\big).
\end{eqnarray*}

As in~\cite{KS}, let $S_{ij}:=\omega_{X_i}^{\otimes -1}\stackrel{L}{\underline{\boxtimes}}\calc_{X^a_{j}}$, and let 
$K^\gamma_{ij}:=\calc^\gamma_{X_i}\stackrel{L}{\underline{\boxtimes}}\omega^\gamma_{X_j^a}$. 
The above computation can be summarized as 
\begin{equation} \label{as3e2}
\calh\calh(\widehat{\cale}_{X_i\times X_j^a},\widehat{\cale}^\gamma_{X_i\times X_j^a})\cong RHom_{\widehat{\cale}_{Z_i\times Z_j^a}}(S_{ij}, K^\gamma_{ij}).
\end{equation}

We obtain the morphism 
\begin{equation}  \label{as3e3}
K^\gamma_{12}\stackrel{L}{\underline{\otimes}}_{\widehat{\cale}_{Z_2}}K^{\gamma}_{23} \stackrel{\cong}{\longrightarrow} 
(\calc^\gamma_{X_1}\stackrel{L}{\underline{\boxtimes}}\omega^{\gamma}_{X_2^a})\stackrel{L}{\underline{\otimes}}_{\widehat{\cale}_{Z_2}}(\calc^{\gamma}_{X_2}\stackrel{L}{\underline{\boxtimes}}\omega_{X_3^a}^{\gamma})\longrightarrow p_{13}^{-1}(\calc^{\gamma}_{X_1}\stackrel{L}{\underline{\boxtimes}}\omega^{\gamma}_{X_3^a})[2d_2]=p^{-1}_{13}(K_{13}^{\gamma})[2d_2].
\end{equation}
For the last arrow in the above composition, note that $\omega^\gamma_{X_2^a}\stackrel{L}{\underline{\otimes}}_{\widehat{\cale}_{Z_2}}\calc^\gamma_{X_2}$ is 
naturally isomorphic to $(\gamma \times 1)_*(\omega_{X_2^a}\stackrel{L}{\underline{\otimes}}_{\widehat{\cale}_{Z_2}}\calc_{X_2})$. Also 
recall that the morphism $\omega_{X_2^a}\stackrel{L}{\underline{\otimes}}_{\widehat{\cale}_{Z_2}}\calc_{X_2} \rightarrow \delta_*{\mathbb C}_{X_2}[2d_2]$ 
is defined by~\cite[Theorem 2.5.7]{KS}. Hence, one obtains a morphism $\omega^\gamma_{X_2^a}\stackrel{L}{\underline{\otimes}}_{\widehat{\cale}_{Z_2}}\calc^\gamma_{X_2}  
\rar \delta^{\gamma}_*\complex_{X_2}[2d_2]$ which induces the last arrow in the above composition. 
The morphism~\eqref{as3e3} induces, by adjunction, a morphism 

\begin{equation} \label{as3e4}
Rp_{13!}(K^{\gamma}_{12} \stackrel{L}{\underline{\otimes}}_{\cale_{Z_2}} K^{\gamma}_{23}) \rightarrow K^{\gamma}_{13} \,\text{.} 
\end{equation}
As explained in the proof of \cite[Proposition 4.2.1]{KS}, there is a natural morphism 
\[
S_{13}\longrightarrow Rp_{13*}(S_{12}\stackrel{L}{\underline{\otimes}}_{\widehat{\cale}_{Z_2}}S_{23}).
\]
With the above two morphisms, we have natural morphisms
\begin{eqnarray*}
&&Rp_{13!}(p^{-1}_{12}\calh\calh(\widehat{\cale}_{X_1\times X_2^a}, \widehat{\cale}_{X_1\times X_2^a}^\gamma)\stackrel{L}{\otimes}p^{-1}_{23}\calh\calh(\widehat{\cale}_{X_2\times X_3^a}, \widehat{\cale}_{X_2\times X_3^a}^\gamma))\\
&&\longrightarrow Rp_{13!}RHom_{\widehat{\cale}_{Z_1\times Z_3^a}}\big(S_{12}\stackrel{L}{\underline{\otimes}}_{\widehat{\cale}_{Z_2}}S_{23}, K_{12}^\gamma\stackrel{L}{\underline{\otimes}}_{\widehat{\cale}_{Z_2}}K_{23}^{\gamma}\big)\\
&& \longrightarrow RHom_{\widehat{\cale}_{Z_1\times Z_3^a}}\big(Rp_{13*}(S_{12}\stackrel{L}{\underline{\otimes}}_{\widehat{\cale}_{Z_2}}S_{23}), Rp_{13!}(K_{12}^\gamma\stackrel{L}{\underline{\otimes}}_{\widehat{\cale}_{Z_2}}K_{23}^{\gamma})\big)\\
&&\longrightarrow RHom_{\widehat{\cale}_{Z_1\times Z_3^a}}(S_{13}, K_{13}^{\gamma})\cong \calh\calh(\widehat{\cale}_{X_1\times X_3^a}, \widehat{\cale}_{X_1\times X_3^a}^{\gamma}).
\end{eqnarray*}
This proves the desired proposition.
\end{proof}

As a corollary, if $X_1=X_3= \text{pt}$ and $X_2=X$, then Proposition \ref{prop:morphism} defines a morphism 
\[
Ra_!\big(\calh\calh(\widehat{\cale}_X, \widehat{\cale}_X^\gamma)\stackrel{L}{\otimes} \calh\calh(\widehat{\cale}_X, \widehat{\cale}_{X}^{\gamma}) \big)\rightarrow \complex_{pt},
\]
where $a:X\to pt$ is the natural map. 
By the adjunction formula, we have 
\begin{corollary}
\label{cor:pairing}Let $X_\reals$ be the underlying real manifold of $X$. There is a canonical morphism $\calh\calh(\widehat{\cale}_X, \widehat{\cale}_X^\gamma)\stackrel{L}{\otimes} \calh\calh(\widehat{\cale}_X, \widehat{\cale}_{X}^{\gamma}) \big)\rightarrow \omega^{top}_{X_\reals}$, where $\omega^{top}_{X_\reals}$ is the topological dualizing complex on $X_\reals$ with coefficients in $\complex$.
\end{corollary}
\begin{remark}Let $HH_{\bullet}(\widehat{\cale}_X, \widehat{\cale}^{\gamma}_X)$ denote the hypercohomology 
${H}^{-\bullet}(X,\calh\calh(\widehat{\cale}_X,\widehat{\cale}^{\gamma}_X))$. 
 We remark that by integration, Corollary \ref{cor:pairing} defines a pairing on $HH_\bullet(\widehat{\cale}_X, \widehat{\cale}_X^\gamma)$, which is a $\gamma$-equivariant generalization of the Fourier-Mukai pairing. 
We hope to discuss more about this pairing in a future publication.  
\end{remark}

We recall that \cite[Definition 3.1.3]{KS} that for $\calk_i\in D^b(\widehat{\cale}_{X_i\times X_{i+1}^a})$ ($i=1,2$), 
\begin{eqnarray*}
\calk_1\circ_{X_2} \calk_2&=&Rp_{13!}(\calk_1\stackrel{L}{\underline{\otimes}_{\widehat{\cale}_2}}\calk_2)\in D^b(\widehat{\cale}_{X_1\times X_3^a}),\\
\calk_1\ast_{X_2} \calk_2&=&Rp_{13*}(\calk_1\stackrel{L}{\underline{\otimes}_{\widehat{\cale}_2}}\calk_2)\in D^b(\widehat{\cale}_{X_1\times X_3^a}).
\end{eqnarray*}
In what follows, we often simplify notations by writing $\circ_2$ for $\circ_{X_2}$ and $\ast_2$ for $\ast_{X_2}$.

We have the following generalization of \cite[Lemma 4.3.3]{KS}.
\begin{lemma}\label{lem:twist-dual}
Let $\gamma\in \Gamma$. Let $\calk$ be a $\gamma$-equivariant element in $D^b_{\text{coh}}(\widehat{\cale}_{X_1\times X_2^a})$. There is a natural morphism in $D^b(\widehat{\cale}_{X_1\times X_1^a})$,
\[
\calk\circ_2\omega^\gamma \circ_2 D'_{\widehat{\cale}}(\calk) \longrightarrow \calc^\gamma_{X_1}.
\]
\end{lemma}
\begin{proof}
By Lemma~\ref{lemma:mor}, we have a morphism in $D^b(\widehat{\cale}_{X_1\times X_2^a\times X_1^a\times X_2})$ 
\[
\gamma_*(\calk) \underline{\boxtimes}D'_{\widehat{\cale}}(\calk) \longrightarrow \calc^{\gamma}_{X_1\times X_2^a}.
\]
Applying the functor $(-) \stackrel{L}{\otimes}_{X_2\times X_2^a} \omega_2^\gamma$, we obtain
\begin{eqnarray*}
&&\big(\calk \stackrel{L}{\underline{\boxtimes}}D'_{\widehat{\cale}}(\calk)\big)\otimes_{X_2\times X_2^a} \omega^\gamma_{X_2}\stackrel{(1)}{\longrightarrow}  \big( \gamma_*(\calk) \stackrel{L}{\underline{\boxtimes}} D'_{\widehat{\cale}}(\calk)\big)\otimes_{X_2\times X_2^a} \omega^\gamma_{X_2}\\
&&\stackrel{(2)}{\longrightarrow} \calc_{X_1\times X_2^a}^\gamma \stackrel{L}{\underline{\otimes}}_{X_2\times X_2^a}\omega^\gamma_{X_2}\stackrel{(3)}{\longrightarrow}  \calc^\gamma_{X_1}\stackrel{L}{\underline{\otimes}}\complex_{X_2}[2\dim(X_2)],
\end{eqnarray*}
Here, in arrow (1), we use the assumption that $\calk$ is $\gamma$-equivariant, i.e. $\gamma$ is a natural element in $Hom(\calk, \gamma_*(\calk))$;  in arrow (2), we have used the morphism in Lemma \ref{lemma:mor}; in arrow (3), we have used the natural isomorphism between 
$\calc^{\gamma}_{X_2^a} \stackrel{L}{\underline{\otimes}}_{X_2 \times X_2^a} \omega^{\gamma}_{X_2}$ and $\delta^{\gamma}_*\complex_{X_2}[2d_2]$; this morphism is obtained by applying the functor $(\gamma \times 1)_*$ 
to the morphism from~\cite[Theorem 2.5.7]{KS}. 
The desired morphism is induced by the above composition of morphisms via adjunction. 
\end{proof}

For $\Lambda$ a closed subset of $X$, let 
$$ HH_{\Lambda}(\widehat{\cale}_X,\widehat{\cale}_X^{\gamma}):= H^0(R\Gamma_{\Lambda}(X; \calh \calh(\widehat{\cale}_X,\widehat{\cale}_X^{\gamma}))) \text{.}$$
Let $\Lambda_{12}$ and $\Lambda_2$ be closed subsets of $X_1\times X_2^a$ and $X_2$. Define $\Lambda_{12}\times_{X_2}\Lambda_2\subset X_1\times X_2$ to be the fiber product of $\Lambda_{12}$ and $\Lambda_2$ over $X_2$, and $\Lambda_{12}\circ \Lambda$ to be $p_{1}(\Lambda_{12}\times_{X_2}\Lambda_2)\subset X_1$. Given a $\gamma$-equivariant kernel $\calk\in D^b_{\text{coh}}(\widehat{\cale}_{X_1\times X_2^a})$ with support $\Lambda_{12}$,   we define the following map 
\[
\Phi_\calk: HH_{\Lambda_2}(\widehat{\cale}_{X_2}, \widehat{\cale}_{X_2}^\gamma)\rightarrow HH_{\Lambda_{12}\circ \Lambda_2}(\widehat{\cale}_{X_1},\widehat{\cale}_{X_1}^\gamma)
\]
via a sequence of compositions,
\begin{eqnarray*}
&&HH_{\Lambda_2}(\widehat{\cale}_{X_2}; \widehat{\cale}^\gamma_{X_2})\cong H^0(R\Gamma_{\Lambda_2}Hom_{X_2\times X_2^a}(\omega_{2}^{\otimes -1}, \calc_2^\gamma))\\
&&\longrightarrow H^0\big(R\Gamma_{\Lambda_{12}\times_{X_2}\Lambda_2}Hom_{X_1\times X_1^a}(\calk\stackrel{L}{\underline{\otimes}_2}\omega^{\otimes -1}_2\circ_2\omega_2\circ_2 D'_{\widehat{\cale}}\calk, \calk\stackrel{L}{\underline{\otimes}_2}\calc^\gamma_2\circ_2\omega_2\circ_2 D'_{\widehat{\cale}}\calk)\big)\\
&&\longrightarrow H^0\big(R\Gamma_{\Lambda_{12}\circ \Lambda_2} Hom_{X_1\times X_1^a}(Rp_*(\calk\stackrel{L}{\underline{\otimes}_2}\omega^{\otimes -1}_2\circ_2\omega_2\circ_2 D'_{\widehat{\cale}}\calk), Rp_{!}(\calk\stackrel{L}{\underline{\otimes}_2}\calc^\gamma_2\circ_2\omega_2\circ_2 D'_{\widehat{\cale}}\calk))\big)\\
&&\cong H^0(\Gamma_{\Lambda_{12}\circ \Lambda_2}Hom_{X_1\times X_1^a}(\calk\ast_2 D'_{\widehat{\cale}}\calk, \calk\circ_2\omega_2^\gamma\circ_2 D'_{\widehat{\cale}}(\calk)))\\
&&\longrightarrow H^0(R\Gamma_{\Lambda_{12}\circ \Lambda_2}Hom_{X_1\times X_1^a}(\omega^{\otimes -1}, \calc_1^\gamma))\cong HH_{\Lambda_{12}\circ \Lambda_2}(\widehat{\cale}_{X_1}; \widehat{\cale}_{X_1}^\gamma),
\end{eqnarray*}
where in the first arrow, we have applied the functor  $\call\mapsto \calk\stackrel{L}{\underline{\otimes}_2}(\call\circ_2 \omega_2\circ_2 D'_{\widehat{\cale}}\calk)$, and in the last arrow we have used Lemma \ref{lem:twist-dual}, and \cite[Lemma 4.3.3]{KS}. 

Let $f:X_2\to X_1$ be a $\gamma$-equivariant  symplectic map. The graph $\Gamma_f$ of $f$ in $X_1\times X_2^a$ is  a Lagrangian submanifold. Denote by $\calb_{\Gamma_f}$ the holonomic $D$-module supported at $\Gamma_f$. It is easy to check from the property of $f$ that $\calb_{\Gamma_f}$ is $\gamma$-equivariant. By Definition \ref{def:lefschetz}, we can define $hh(f,\gamma)=hh^{\gamma}(\calb_{\Gamma_f},\gamma)\in H^0_{\Gamma_f}(X_1\times X_2; \calh\calh(\widehat{\cale}_{X_1\times X_2^a}, \widehat{\cale}_{X_1\times X_2^a}^\gamma))$. 

The proof of \cite[Lemma 4.3.4]{KS} may be generalized word for word to give the following result. 
\begin{proposition}
\label{prop:push forward} The following morphisms are equal,
$$
\Phi_{\calb_{\Gamma_f}}=hh(f, \gamma)\circ : HH_{\Lambda_2}(\widehat{\cale}_{X_2}, \widehat{\cale}_{X_2}^\gamma)\to HH_{\Lambda_{12} \circ \Lambda_2}(\widehat{\cale}_{X_1}, \widehat{\cale}_{X_1}^\gamma).
$$
\end{proposition}

\begin{proof}
Let $\alpha_2$ be a class in $HH(\widehat{\cale}_{X_2}, \widehat{\cale}^\gamma_{X_2})$. By the isomorphism 
\[
\calh\calh(\widehat{\cale}_X, \widehat{\cale}_X^\gamma)\cong \delta^{-1}_X R\calh om_{\widehat{\cale}_{X\times X^a}}(D'_{\widehat{\cale}_{X^a\times X}} (\calc_{X^a}), \calc^\gamma_X)\cong \delta^{-1}_X R\calh om_{\widehat{\cale}_{X\times X^a}}(\omega_X^{\otimes -1}, \calc^\gamma_X),
\]
we can regard $\alpha_2$ as a morphism $\alpha_2: \omega^{\otimes -1}_{X_2}\longrightarrow \calc^\gamma_{X_2}$ in the derived category of sheaves of $\complex$-vector spaces on $X_{22^a}:=X_2\times X_2^a$. Similarly, we can regard the element $\alpha=hh(f,\gamma)$ in $HH(\widehat{\cale}_{X_1\times X_2^a}, \widehat{\cale}_{X_1\times X_2^a}^\gamma)$ as a morphism $\alpha:\omega^{\otimes -1}_{X_1\times X_2^a}\to \calc^\gamma_{X_1\times X_2^a}$ in the derived category of sheaves on $X_{11^a22^a}:=X_1\times X_1^a\times X_2\times X_2^a$. By Lemma~\ref{lks4.1.4}, $\alpha$ is given by a composite of morphisms 
$$\omega_{12^a}^{\otimes -1} \xrightarrow{\text{\cite[Lemma 4.1.1(i)]{KS}}} \calk \stackrel{L}{\underline{\boxtimes}} D'\calk \stackrel{\bar{\alpha}}{\longrightarrow} \calc^{\gamma}_{12^a}$$
in the derived category of sheaves on $X_{11^a22^a}$, where $\calb_{\Gamma_f}$ is denoted by $\calk$. The element $\Phi_{\calb_{\Gamma_f}}(\alpha)$ is an element represented by the morphism 
\[
\begin{split}
&\omega_{1}^{\otimes -1}\to \calk \ast_2 D'_{\widehat{\cale}} \calk\to Rp_{1*}\big( \calk \stackrel{L}{\underline{\otimes}}_2 (\omega^{\otimes -1}_2 \circ_2 \omega_2\circ_2 D'_{\widehat{\cale}}\calk)\big)\\
& \stackrel{\alpha_2}{\longrightarrow}Rp_{1!}\big(\calk \stackrel{L}{\underline{\otimes}_2}(\calc_2^\gamma \circ_2 \omega_2\circ_2 D'_{\widehat{\cale}}\calk)\big)\to Rp_{1!}\big(\calk \circ_2 \omega^\gamma \circ D'_{\widehat{\cale}}\calk\big)\xrightarrow{\text{Lemma \ref{lem:twist-dual}}}\calc^\gamma_{X_1}
\end{split}
\]
in the derived category of sheaves on $X_1$. The following commutative diagram in the category $D^b(\widehat{\cale}_{11^a}\boxtimes \complex_{X_2\times X_2^a})$ directly generalizes a subdiagram of a diagram appearing in the one in the proof of \cite[Lemma 4.3]{KS} (see \cite[Page 111]{KS}). The only genuine change in following diagram from the one in the proof of \cite[Lemma 4.3]{KS} is to change $\otimes_{22^a}\calc_2$ to $\otimes_{22^a}\calc^\gamma_2$. We also point out to the reader that the last row in the diagram below is written in a different (though equivalent) way than the corresponding row in \cite[Page 111]{KS} (modulo the above mentioned change from $\otimes_{22^a}\calc_2$ to $\otimes_{22^a}\calc^\gamma_2$).  
{\small
\[
\begin{diagram}
\node{p^{-1}_{11^a}\omega^{\otimes -1}_1}\arrow{s,t, }{}\\
\node{\big(\omega^{\otimes -1}_1\underline{\boxtimes}\calc_{2^a}\big)\stackrel{L}{\underline{\otimes}}_{22^a}\omega_2^{\otimes-1}}\arrow{s,t,}{}\arrow{e,t,t}{\alpha_2}\node{\big( \omega^{\otimes -1}_1\underline{\boxtimes} \calc_{2^a}\big)\stackrel{L}{\underline{\otimes}}_{22^a}\calc_2^\gamma}\arrow{s,t,}{}\arrow{e,t,}{}\node{\big((\calk\underline{\boxtimes}D'\calk)\circ_{2^a}\omega_{2^a}\big)\stackrel{L}{\underline{\otimes}}_{22^a}\calc_2^\gamma}\arrow{s,t,l}{\bar{\alpha}}\\
\node{(\calk \underline{\boxtimes} D'_{\widehat{\cale}}\calk) \stackrel{L}{\underline{\otimes}}_{22^a} (\omega_2^{\otimes -1} \circ_2 \omega_2)}\arrow{e,t,t}{\alpha_2}\node{(\calk \underline{\boxtimes} D'_{\widehat{\cale}}\calk) \stackrel{L}{\underline{\otimes}}_{22^a} (\calc_2^{\gamma} \circ_2 \omega_2)}\arrow{e,t,t}{\text{Lemma \ref{lem:twist-dual}}}\node{\calc_1^\gamma\boxtimes \complex_{X_2}[2\text{dim}(X_2)]}
\end{diagram}
\]
}
By adjunction, the map from $p^{-1}_{11^a}\omega_1^{\otimes -1}$ to $\calc_1^\gamma\boxtimes \complex_{X_2}[2\text{dim}(X_2)]$ via the composition of the upper row with the right column is $\alpha\circ \alpha_2$ while the map via the composition of the left column with the lower row is
 $\Phi_{\calb_{\Gamma_f}}(\alpha_2)$. This gives the desired equality of morphisms. 
\end{proof}

\begin{remark}
 It is interesting to compare Proposition~\ref{prop:push forward} with~\cite[Theorem 5.4]{G}, which is the direct image theorem for the Lefschetz class constructed in~\cite{G}. The integral transform $\Phi_{\calb_{\Gamma_f}}$ in Proposition~\ref{prop:push forward} 
corresponds to an honest morphism $f:X_1 \rightarrow X_2$ of complex manifolds in~\cite{G}. On the other hand, the holomorphic diffeomorphisms 
$\gamma_{X_1}$ and $\gamma_{X_2}$ that appear\footnote{$\gamma_{X_1}$ (resp., $\gamma_{X_2}$) denotes the holomorphic diffeomorphism 
$\gamma$ acting on $X$ (resp., $X_2$)} in  Proposition~\ref{prop:push forward} correspond to to a pair of integral transforms, one on $X_1$ and the other on $X_2$ satisfying certain compatibility criteria with respect to $f$. 

It would be interesting to generalize the material in this and the previous subsection (Proposition~\ref{prop:push forward} in particular)
 to the case when $\gamma$ acts on $X_1$ as well as $X_2$ 
by integral transforms rather than holomorphic diffeomorphisms. The approach here seriously utilizes the fact that 
$\gamma$ acts by holomorphic diffeomorphisms, making such a generalization non-trivial. Such a generalization would yield a more general 
direct image theorem for the Hochschild Lefschetz class than~\cite[Theorem 5.4]{G}. Further, when combined with an 
understanding of trace densities, such a result would yield a (possibly simpler) approach to generalizations of 
~\cite[Theorem 5.4]{G} itself.  
\end{remark}

\subsection{Orbifold Hochschild and Chern class}  \label{subsec:class} Let $Q_X$ (and $Q_M$) be the orbifold defined by the quotient $X/\Gamma$ (and $M/\Gamma$) for $X=T^*M$ and let 
 $\mathfrak{q}:X\to Q_X$ be the canonical quotient map. Define a sheaf of algebras $\cala$ on $Q_X$ by 
\[
\cala(U):= \widehat{\cale}_X(\mathfrak{q}^{-1}(U))\rtimes \Gamma, 
\]
for any (sufficiently small) open subset $U\subset Q_X.$ In the above definition, $\Gamma$ acts on $\mathfrak{q}^{-1}(U)$, and therefore acts on the algebra $\widehat{\cale}_X(\mathfrak{q}^{-1}(U))$. $ \widehat{\cale}_X(\mathfrak{q}^{-1}(U))\rtimes \Gamma$ is the associated crossed product algebra.  

Let $\calm$ be a good $\Gamma$-equivariant coherent $\cald_M$-module and $\widehat{\calm}$  the corresponding $\Gamma$-equivariant $\widehat{\cale}_X$-module. Define $\mathfrak{M}$ to be a sheaf on $Q_X$ by 
\[
\mathfrak{M}(U):=\widehat{\calm}(\mathfrak{q}^{-1}(U)), 
\]
for any (sufficiently small) open subset $U\subset Q_X.$  On an open subset $U\subset Q_X$, both $\gamma$ and $ \widehat{\cale}_X(\mathfrak{q}^{-1}(U))$ naturally act on $\widehat{\calm}(\mathfrak{q}^{-1}(U))$ with the appropriate commutation relation between these actions. This equips $\mathfrak{M}$ with a natural $\cala$-module structure.
It is not difficult to check that if $\calm$ is a good coherent $\cald_M$-module, 
$\mathfrak{M}$ is a good coherent $\cala$-module.  
We apply the following theorem to construct the Hochschild class 
$hh^\cala_{\text{supp}(\calm)/\Gamma,i}(\mathfrak{M})$ and the cyclic
 class $ch^\cala_{\text{supp}(\calm)/\Gamma, i}(\mathfrak{M})$ of a perfect 
complex $\mathfrak{M}$ of $\cala$-modules, where $\text{supp}(\calm)/\Gamma$ is the 
support of $\mathfrak{M}$ in $Q_X$.

\begin{theorem}\label{thm:bnt}(\cite[Theorem 2.1.1.]{BNT})
Let $Q$ be a topological space and $Z$ a closed subset of $Q$. Let $\cala$ be a sheaf of algebras on $Q$ such that there is a global section 
$1\in \Gamma(Q; \cala)$ which restricts to $1_{\cala_x}$ for all $x\in Q$. 
Let $\calh\calc^-(\cala)$ (resp., $\calh\calh(\cala)$) be  the sheaf of negative cyclic (resp., Hochschild) homologies\footnote{As in Section~\ref{sec:notations}, we abuse terminology here:  $\calh\calc^-(\cala)$ and $\calh\calh(\cala)$ are objects in the derived category of sheaves of $\complex$-vector spaces on $Q$. Also, when $Q=X:=T^*M$ as in Section~\ref{sec:notations} and when $\cala=\widehat{\cale}_X$, $\calh\calh(\cala)$ as defined in~\cite{BNT} is isomorphic to $\calh\calh(\cala)$ as defined in Section~\ref{sec:notations}.} of $\cala$. Denote by $K^i_Z(\cala)$ the $i$-th K-group of the category of perfect complexes of $\cala$-modules which are acyclic outside $Z$. There exists the cyclic class ${\rm ch}^\cala_{Z,i}:K^i_Z(\cala)\to H^{-i}_Z(Q; \calh\calc^- (\cala))$ and the Hochschild class $hh^\cala_{Z,i}:K^i_Z(\cala)\to H^{-i}_Z(Q; \calh\calh(\cala))$ such that 
\begin{itemize}
\item the composition 
\[
K^i_Z(\cala)\xrightarrow{{\rm ch}^\cala_{Z,i}}H^{-i}_Z(Q; \calh\calc^-(\cala))\longrightarrow H^{-i}_Z(Q; \calh\calh(\cala))
\]
coincides with $hh^\cala_{Z,i}$;
\item for a perfect complex $\calf^\bullet$ of $\cala$-modules supported on $Z$ the Hochschild class 
$$hh^{\cala}_{Z,0}(\calf^\bullet)\in H^0_Z(Q; \calh\calh(\cala))$$ 
coincides with the composition 
\[
k\xrightarrow{1\mapsto id} \calr \calh om_\cala(\calf^\bullet, \calf^\bullet)\stackrel{\cong}{\longleftarrow} (\calr\calh om_\cala(\calf^\bullet, \cala)\boxtimes \calf^\bullet)\otimes^L_{\cala\otimes \cala^{op}}\cala\xrightarrow{ev\otimes id}\cala\otimes^L_{\cala\otimes \cala^{op}}\cala.
\]
\end{itemize}
\end{theorem}

Applying Theorem \ref{thm:bnt}, for a good coherent $\Gamma$-equivariant  $\cald_M$-module $\calm$, we have a well defined Hochschild class $hh^\cala_{Z, 0}(\mathfrak{M})\in H^{0}_Z(Q_X; \calh\calh(\cala))$ and cyclic class ${\rm ch}^\cala_{Z,0}(\mathfrak{M})\in H^{0}_Z(Q_X; \calh\calc^-(\cala))$ where $Z=\text{supp}(\calm)/\Gamma$ and $\cala$ is the sheaf of crossed product algebras defined by $\cala(U):=\widehat{\cale}_X(\mathfrak{q}^{-1}(U))\rtimes \Gamma$ (for sufficiently smal open sets $U$ in $Q_X$).  

On $Q_X$, we can also consider the sheaf of algebras $\widehat{\cale}_{Q_X}$ defined by $$\widehat{\cale}_{Q_X}(U):=\widehat{\cale}_X(\mathfrak{q}^{-1}(U))^\Gamma \quad \text{(for $U$ sufficiently small)}.$$ 
Here, $\widehat{\cale}_X(\mathfrak{q}^{-1}(U))^\Gamma$ is the space of $\Gamma$-invariant sections of $\widehat{\cale}_X(\mathfrak{q}^{-1}(U))$. Similarly, we consider the good coherent $\widehat{\cale}_{Q_X}$-module $\widehat{\calm}_{Q_X}$ defined by $$\widehat{\calm}_{Q_X}(U):=\widehat{\calm}(\mathfrak{q}^{-1}(U))^\Gamma.$$ 
Applying Theorem \ref{thm:bnt} to $\widehat{\cale}_{Q_X}$ and $\widehat{\calm}_{Q_X}$, we obtain 
the Hochschild and cyclic classes $$hh^{\widehat{\cale}_{Q_X}}_{Z, 0}(\widehat{\calm}_{Q_X})\in H^{0}_{Z}(Q_X; \calh\calh (\widehat{\cale}_{Q_X}))\,\, \text{and } ch^{\widehat{\cale}_{Q_X}}_{Z,0}(\widehat{\calm}_{Q_X})\in H^{0}_Z(Q_X; \calh \calc (\widehat{\cale}_{Q_X})),$$ where $Z=\text{supp}(\calm)/\Gamma$.

Consider the global section $$e=\frac{1}{|\Gamma|}\sum_{\gamma\in \Gamma} \gamma\in \Gamma(\cala)$$ of the sheaf $\cala$. 
It is easy to check that $e$ is a projection. Define a sheaf $\calv$ of $\cala$-$\widehat{\cale}_{Q_X}$-bimodules by $$\calv(U):=\widehat{\cale}_X(\mathfrak{q}^{-1}(U))\rtimes \Gamma e|_U,$$ and a sheaf $\calw$ of $\widehat{\cale}_{Q_X}$-$\cala$-bimodules  by $$\calw(U):=e|_U\widehat{\cale}_X(\mathfrak{q}^{-1}(U))\rtimes \Gamma.$$ 
 $\calv$ and $\calw$ are Morita equivalence bimodules between $\cala$ and $\widehat{\cale}_{Q_X}$. Under this Morita equivalence, $\widehat{\calm}_{Q_X}$ corresponds to the sheaf $\mathfrak{M}$.  With the explicit bimodules $\calv$ and $\calw$, 
we can easily check that under the Morita isomorphism between the Hochschild and cyclic homologies of 
$\widehat{\cale}_{Q_X}$ and those of $\cala$, the Hochschild and cyclic classes of $\widehat{\calm}_{Q_X}$ 
are identified with those of $\mathfrak{M}$. 
\[
\begin{split}
hh^\cala_{Z,0}(\mathfrak{M})&=hh^{\widehat{\cale}_{Q_X}}_{Z,0}(\widehat{\calm}_{Q_X})\in H^{0}_Z(Q_X; \calh\calh(\cala))\cong H^{0}_{Z}(Q_X; \calh\calh (\widehat{\cale}_{Q_X})),\qquad \\
{\rm ch}^\cala_{Z,0}(\mathfrak{M})&={\rm ch}^{\widehat{\cale}_{Q_X}}_{Z,0}(\widehat{\calm}_{Q_X})\in H^{0}_Z(Q_X; \calh\calc^{-}(\cala))\cong H^{0}_{Z}(Q_X;  \calh\calc^{-}(\widehat{\cale}_{Q_X})). 
\end{split}
\] 
The Hochschild and cyclic homology of $\cala$ is computed in \cite{DE} and \cite{NPPT}
\[
\mu^{\cala}:\calh\calh(\cala)\cong (\oplus_\gamma \complex_{X^\gamma}[\dim(X^\gamma)])^\Gamma,\qquad \mu^\cala: \calh\calc^-(\cala)\cong  (\oplus_{\gamma,\bullet\geq 0 }\complex_{X^\gamma}[\dim(X^\gamma)-2(\bullet)])^\Gamma,
\]
where $\gamma\in \Gamma$ acts on $\oplus_\gamma \complex_{X^\gamma}[\dim(X^\gamma)]$ mapping the $\alpha$-component to the $\gamma\alpha\gamma^{-1}$-component.  

Let $IQ_X$ be the inertia orbifold associated to $Q_X$, defined by $$IQ_X:=(\sqcup_{\gamma\in \Gamma} X^\gamma)/\Gamma,$$ where $\gamma\in \Gamma$ acts on $\sqcup_\gamma X^\gamma$ by mapping $(\alpha, x)$ with $\alpha(x)=x$ to $(\gamma\alpha\gamma^{-1}, \gamma(x))$. Let $\iota_{IQ_X}:IQ_X\to Q_X$ be the natural map 
defined by forgetting the group element. Thus, we have 
\[
(\oplus_{\gamma\in\Gamma} \complex_{X^\gamma}[\dim(X^\gamma)])^\Gamma=\iota_{IQ_X, *}\complex[\dim(IQ_X)]. 
\]

\begin{definition}\label{dfn:class}
The {\em orbifold Euler class} $\text{eu}_{Q_X}(\calm)$ (resp., the {\em orbifold Chern class} $\text{ch}_{Q_X}(\calm)$) of a good $\Gamma$-equivariant coherent $\cald_M$-module $\calm$ is defined to be the images of $hh^\cala_{Z,0}(\mathfrak{M})$ (resp., ${\rm ch}^\cala_{Z,0}(\mathfrak{M})$) in $H^{0}_Z(IQ_X; \complex[\dim(IQ_X)])$ 
(resp., $\oplus_{n \geq 0} H^0_Z(IQ_X; \complex[\dim(IQ_X)-2n])$). 
\end{definition}

\begin{remark} In \cite{BNT}, the classes $hh^\cala_{Z,0}$ and ${\rm ch}^\cala_{Z,0}$ are called the Euler and the Chern class respectively. Here, we distinguish them from their images in the $H^\bullet_Z(IQ_X, \complex[\dim(IQ_X)])$, which are closer to the classical Euler and Chern characters.  
\end{remark}

In the remaining part of this section, we will explain the relation between the Hochschild Lefschetz class in Definition \ref{def:lefschetz} and orbifold Hochschild class in Theorem \ref{thm:bnt}.

We observe that  $\sum_{\gamma\in \Gamma} hh^\gamma(\calm, \gamma)\in \bigoplus_{\gamma \in \Gamma} 
H^0_{\text{supp}(\calm)}(X; \calh\calh(\widehat{\cale}_X, \widehat{\cale}_X^\gamma))$ is invariant under 
the action of $\Gamma$ on $\bigoplus_{\gamma \in \Gamma} 
H^0_{\text{supp}(\calm)}(X; \calh\calh(\widehat{\cale}_X, \widehat{\cale}_X^\gamma))$ induced by the conjugation 
action $\alpha\mapsto \gamma\alpha\gamma^{-1}$ of $\Gamma$ on itself. Consider
\[
\begin{split}
\widetilde{hh}^{Q_X}_{Z,0}(\calm):= \frac{1}{|\Gamma|}\sum_{\gamma\in\Gamma} hh^\gamma(\calm, \gamma)&\in \Big(\bigoplus_{\gamma\in \Gamma} H^0_{\text{supp}(\calm)}(X; \calh\calh(\widehat{\cale}_X,\widehat{\cale}_X^\gamma))\Big)^\Gamma\\
&\cong H^0_{\text{supp}(\calm)}\Big(Q_X;\big( \bigoplus_{\gamma\in\Gamma} \calh\calh(\widehat{\cale}_X,\widehat{\cale}_X^\gamma)\big)^\Gamma \Big). \\
\end{split}
\]
Here, by abuse of notation, we also use the symbol $\widehat{\cale}_X$ to denote the sheaf $U \mapsto \widehat{\cale}_X(\mathfrak{q}^{-1}(U))$ of algebras on 
the orbifold $Q_X$. Note that the sheaf $\widehat{\cale}_X$ is a sheaf of algebras on $Q_X$ with a (local) $\Gamma$-action and that $\widehat{\cale}_{Q_X}\,=\, \widehat{\cale}_X^{\Gamma}$. 
The Hochschild homology $H^0_Z(Q_X,\calh\calh(\cala))$ is naturally isomorphic to  $\Big(\oplus_{\gamma \in \Gamma} 
H^0_{\text{supp}(\calm)}(X; \calh\calh(\widehat{\cale}_X, \widehat{\cale}_X^\gamma))\Big)^{\Gamma}$ (see e.g. \cite{DE}). 
 Identifying  $\Big(\oplus_{\gamma \in \Gamma} 
H^0_{\text{supp}(\calm)}(X; \calh\calh(\widehat{\cale}_X, \widehat{\cale}_X^\gamma))\Big)^{\Gamma}$  with $H^0_Z(Q_X,\calh\calh(\cala))$ using this isomorphism, the following equality holds.
\begin{theorem}\label{thm:class}
\[
\widetilde{hh}^Q_{Z,0}(\calm)=hh^\cala_{Z,0}(\mathfrak{M}). 
\]
\end{theorem}
We shall now sketch the proof of Theorem~\ref{thm:class}, leaving details to the interested reader.
\begin{proof}[Sketch of proof]
In what follows, $\widehat{\cale}:=\widehat{\cale}_X$ is thought of as a sheaf of algebras on $Q_X$.  Let $\gamma_*(\mathfrak{M})$ denote the sheaf $\mathfrak{M}$ on $Q_X$, whose $\widehat{\cale}$-module 
structure is twisted by $\gamma$ like $\gamma_*(M)$. 

Define $L:\calh om_{\widehat{\cale}\rtimes \Gamma}(\mathfrak{M}, \mathfrak{M})\to \big(\bigoplus_{\gamma\in\Gamma} \calh om_{\widehat{\cale}}(\mathfrak{M}, \gamma_*(\mathfrak{M}))\big)^\Gamma$ by 
\[
L(F):=\frac{1}{|\Gamma|}\sum_{\gamma\in \Gamma} \gamma\circ F. 
\]
Here, $\gamma \circ F(m) = \gamma^{-1}(F(\gamma(m)))$ for $m\,\in\,\Gamma(U,\widehat{\cale}_X)$. 
Define $$\tilde{L}:\calh om_{\widehat{\cale}\rtimes \Gamma}(\mathfrak{M}, \widehat{\cale}\rtimes \Gamma)\otimes_{\widehat{\cale}\rtimes \Gamma}\mathfrak{M}\to \big(\bigoplus_{\gamma\in\Gamma} \calh om_{\widehat{\cale}}(\mathfrak{M}, \widehat{\cale})\otimes_{\widehat{\cale}}\gamma_*(\mathfrak{M})\big)^\Gamma$$ by
\[
\tilde{L}(F\otimes_{\widehat{\cale}\rtimes \Gamma} m):=\frac{1}{|\Gamma|}\sum_\gamma\sum_\alpha F_\alpha\otimes \gamma^{-1}(\alpha(m)),
\]
where $F_\alpha\in \calh om_{\widehat{\cale}}(\mathfrak{M}, \widehat{\cale})$ is defined by $F=\sum_{\alpha} F_\alpha \otimes \alpha \in \calh om_{\widehat{\cale}}(\mathfrak{M}, \widehat{\cale})\otimes_{\widehat{\cale}}(\widehat{\cale} \rtimes \Gamma)$ and 
$ \gamma^{-1}(\alpha(m))$ is viewed as a section of $\gamma_*(\mathfrak{M})$.  Also recall that the action of an element $g \in \Gamma$ on 
$\bigoplus_{\gamma\in \Gamma} \calh om_{\widehat{\cale}}(\mathfrak{M}, \widehat{\cale})\otimes_{\widehat{\cale}}\gamma_*(\mathfrak{M})$ 
takes a section of the form $F(\mbox{--}) \otimes m$ to $g(F(g^{-1}(\mbox{--}))) \otimes g.m$. The morphisms denoted by $\mu$ in the diagram below are the obvious ``evaluation'' maps: 
\[
\begin{diagram}
\node{\calh om_{\widehat{\cale}\rtimes \Gamma}(\mathfrak{M}, \mathfrak{M})}\arrow{s,l,t}{L}\node{\calh om_{\widehat{\cale}\rtimes \Gamma}(\mathfrak{M}, \widehat{\cale}\rtimes \Gamma)\otimes_{\widehat{\cale}\rtimes \Gamma}\mathfrak{M}}\arrow{w,l,t}{\mu} \arrow{s,l,t}{\tilde{L}}\\ 
\node{\big(\bigoplus_{\gamma\in\Gamma} \calh om_{\widehat{\cale}}(\mathfrak{M}, \gamma_*(\mathfrak{M}))\big)^\Gamma}\node{\big(\bigoplus_{\gamma\in \Gamma} \calh om_{\widehat{\cale}}(\mathfrak{M}, \widehat{\cale})\otimes_{\widehat{\cale}}\gamma_*(\mathfrak{M})\big)^\Gamma}\arrow{w,t,t}{\mu}
\end{diagram}
\] 
It is straightforward to check that the following diagram commutes. 

Define $\bar{L}: \big(\calh om_{\widehat{\cale}\rtimes \Gamma}\big(\mathfrak{M},\widehat{\cale}\rtimes \Gamma\big)\boxtimes \mathfrak{M}\big)\otimes_{\cala_{Q\times Q^{a}}}( \widehat{\cale}\rtimes \Gamma)\to \Big(\big(\bigoplus_{\gamma\in\Gamma} \calh om_{\widehat{\cale}}(\mathfrak{M}, \widehat{\cale})\boxtimes\gamma_*(\mathfrak{M})\big)\otimes_{\widehat{\cale}\otimes \widehat{\cale}^{a}}\widehat{\cale} \Big)^\Gamma$ by
\[
\bar{L}\big(((F_\alpha \otimes\alpha)\boxtimes m)\otimes (d\otimes \beta)\big):=\frac{1}{|\Gamma|}\sum_{\gamma\in \Gamma} (F_\alpha \boxtimes \gamma^{-1}\alpha\beta(m))\otimes \alpha(d). 
\]
Here, $\gamma^{-1}\alpha\beta(m)$ is viewed as a section of $\gamma_*(\mathfrak{M})$. Also recall that the action of an element 
$g \in \Gamma$ on $\bigoplus_\gamma \calh om_{\widehat{\cale}}(\mathfrak{M}, \widehat{\cale})\boxtimes\gamma_*(\mathfrak{M})$ takes a section $(F(\mbox{--}) \boxtimes m) \otimes f$ to $(g(F(g^{-1}(\mbox{--}))) \boxtimes g.m) \otimes g.f$. The following diagram commutes:
\[
\begin{diagram}
\node{\calh om_{\widehat{\cale}\rtimes \Gamma}(\mathfrak{M}, \widehat{\cale}\rtimes \Gamma)\otimes_{\widehat{\cale}\rtimes \Gamma}\mathfrak{M}}\arrow{s,l,t}{\tilde{L}}\node{\big(\calh om_{\widehat{\cale}\rtimes \Gamma}\big(\mathfrak{M},\widehat{\cale}\rtimes \Gamma\big)\boxtimes \mathfrak{M}\big)\otimes_{\cala_{Q\times Q^{a}}}( \widehat{\cale}\rtimes \Gamma)}\arrow{w,t,t}{\cong}\arrow{s,l,t}{\bar L}\\ 
\node{\big(\bigoplus_{\gamma\in \Gamma} \calh om_{\widehat{\cale}}(\mathfrak{M}, \widehat{\cale})\otimes_{\widehat{\cale}}\gamma_*(\mathfrak{M})\big)^\Gamma} \node{\Big(\big(\bigoplus_{\gamma\in \Gamma} \calh om_{\widehat{\cale}}(\mathfrak{M}, \widehat{\cale})\boxtimes\gamma_*(\mathfrak{M})\big)\otimes_{\widehat{\cale}\otimes \widehat{\cale}^{a}}\widehat{\cale} \Big)^\Gamma}\arrow{w,t, t}{\cong}
\end{diagram}
\]
Define $\hat{L}:(\widehat{\cale}\rtimes \Gamma)\otimes_{\cala_{Q\times Q^{a}}}(\widehat{\cale}\rtimes \Gamma)\to \big(\bigoplus_{\gamma\in\Gamma} \gamma_*(\widehat{\cale}) \otimes \widehat{\cale}\big)^\Gamma$ by
\[
\hat{L}((e_0\otimes \alpha)\otimes (e_1\otimes \beta))=\frac{1}{|\Gamma|}\sum_{\gamma\in\Gamma}\gamma(e_0)\otimes \gamma\alpha(e_1).
\]
Here, $\gamma(e_0)$ is viewed as a section of $(\gamma\alpha\beta\gamma^{-1})_*(\widehat{\cale})$ and $\gamma\alpha(e_1)$ is viewed as a section of $\cale$. Also recall that the action of an element $g \in \Gamma$ maps a section $e \otimes f$ of $\gamma_*(\widehat{\cale}) \otimes\cale$ to the section $ge \otimes gf$ of $(g\gamma g^{-1})_*(\widehat{\cale}) \otimes \cale$. We further have the following commutative diagram:
\[
\begin{diagram}\node{\big(\calh om_{\widehat{\cale}\rtimes \Gamma}\big(\mathfrak{M},\widehat{\cale}\rtimes \Gamma\big)\boxtimes \mathfrak{M}\big)\otimes_{\cala_{Q\times Q^{a}}}( \widehat{\cale}\rtimes \Gamma)}\arrow{s,l,t}{\bar{L}}\arrow{e,t,t}{ev\otimes id}\node{(\widehat{\cale}\rtimes \Gamma)\otimes_{\cala_{Q\times Q^{a}}}(\widehat{\cale}\rtimes \Gamma)}\arrow{s,r,t}{\hat{L}}\\
\node{\Big(\big(\bigoplus_{\gamma\in \Gamma} \calh om_{\widehat{\cale}}(\mathfrak{M}, \widehat{\cale})\boxtimes\gamma_*(\mathfrak{M})\big)\otimes_{\widehat{\cale}\otimes \widehat{\cale}^{a}}\widehat{\cale} \Big)^\Gamma}\arrow{e,t,t}{ev\otimes id}\node{\big(\bigoplus_{\gamma\in \Gamma} \gamma_*(\widehat{\cale})\otimes \widehat{\cale}\big)^\Gamma}
\end{diagram}
\]
Combining the above three commutative diagrams gives us the desired theorem: indeed, the image of $\text{id} \,\in Hom_{\widehat{\cale} \rtimes \Gamma}(\mathfrak{M},\mathfrak{M})$ in $H^0(Q_X, (\widehat{\cale} \rtimes \Gamma) \stackrel{L}{\otimes}_{\cala_{Q_X \times Q_X^{a}}}  (\widehat{\cale} \rtimes \Gamma))$ under the morphism induced by the upper horizontal arrows in the above three diagrams is $hh^\cala_{Z,0}(\mathfrak{M})$, while the image of $\gamma \in Hom_{\widehat{\cale}}(\mathfrak{M},\gamma_*(\mathfrak{M}))$ in $H^0(Q_X; \calh\calh(\widehat{\cale},\widehat{\cale}^{\gamma}))$ under the morphism induced by the lower arrows in the above three diagrams is  $hh^{\gamma}(\calm,\gamma)$. 
\end{proof}

\begin{remark}\label{rmk:orbifold}  
The orbifold Euler class in Definition \ref{dfn:class} has a direct generalization to general orbifolds other than global quotient orbifolds, i.e. orbifolds 
of the form $M/\Gamma$. This generalization is obtained by working 
with the sheaf of invariant differential operators as explained in \cite{FT}. 
 Our orbifold Riemann-Roch theorem in the next section also generalizes to this setting. 
We will leave the details of this generalization to the reader.  
\end{remark}
\section{Euler class on an orbifold}\label{sec:euler}
In this section, we prove an orbifold Riemann-Roch theorem for the orbifold Euler class $\text{eu}_{Q_X}(\calm)$ of a good $\Gamma$-equivariant coherent $\cald_M$-module introduced in  Definition \ref{dfn:class}.  Our main strategy is to generalize the method developed by Bressler, Nest, and Tsygan \cite{BNT} to  orbifold setting. 

\subsection{Deformation quantization}
Our strategy to compute the $\text{eu}_{Q_X}(\calm)$ is to transfer the computation to a more flexible context: that of
 deformation quantization modules.  Closely related to $\widehat{\cale}_X$ and $\widehat{\calm}$, 
is the (sheaf of) deformation quantization algebra(s)
 $\widehat{\calw}_{X}(0)$ on $X=T^*M$ over the ring $\complex[[\hbar]]$ constructed in \cite{PS}
 modeled on (the sheaf of) negative order formal microdifferential operators. 
Let $\widehat{\calw}_X$ be the localization of $\widehat{\calw}(0)$ defined by 
$$
\widehat{\calw}_X:=\widehat{\calw}(0)\otimes_{\complex[[\hbar]]}\complex((\hbar)).
$$ 
The sheaves of algebras $\cald_M$, $\widehat{\cale}_X$, and $\widehat{\calw}_{X}$ are naturally 
related to one another by the inclusions
\[
\pi_M^{-1}\cald_M\hookrightarrow \widehat{\cale}_X\hookrightarrow \widehat{\calw}_X.
\]
Following~\cite{KS} we consider the functor 
$$(\cdot)^W: Mod(\cald_M)\to Mod(\widehat{\calw}_X)\,,$$ 
$$ \calm \mapsto \calm^W:=\widehat{\calw}_X\otimes_{\pi^{-1}_M\cald_M}\pi^{-1}_M \calm\,\text{.}$$
 By~\cite[Proposition 6.4.1]{KS}, this functor is exact, faithful, and preserves properties such as coherence and goodness. 

Note that the Lefschetz class of a good coherent $\cald_M$-module with values in the Hochschild homology of $\widehat{\calw}_X$ can be defined in exactly the same way as how the Lefschetz class of a good coherent $\cald_M$-module with values in the Hochschild homology of $\widehat{\cale}_X$ is defined in Section~\ref{sec:lefschetz}. 

To be precise, given a good coherent $\cald_M$-module $\calm,$ we consider the associated good coherent $\widehat{\calw}_X$-module $\calm^W$. As is explained in Section \ref{sec:notations}, $\gamma_*$ defines a natural functor from $D^b(\widehat{\calw}_X)$ (and $D^b_{\text{coh}}(\widehat{\calw}_X)$) to itself. Let $\calc_X^{W}:=\delta_{X,*} \widehat{\calw}_X$, viewed as a 
$\widehat{\calw}_{X \times X^a}$-module. Similarly, let $\calc_X^{\gamma,W}:= \delta_{X,*}^{\gamma} 
\widehat{\calw}_X$. 
 We have a natural morphism analogous to that of Lemma~\ref{lemma:mor}:
\[
\gamma_*( {\calm}^W)\stackrel{L}{\underline{\boxtimes}}D'_{\widehat{\calw}_X}(\calm^W)\rightarrow \calc_X^{\gamma,W}\,.
\] 
For  $u\in Hom_{{\cald}_M}(\calm, \gamma_*(\calm))$, the definition of the Hochschild Lefschetz class $hh^{\gamma, W}(\calm,u)$ of a 
good coherent $\cald_M$ module $\calm$ is completely analogous to Definition~\ref{def:lefschetz}. Indeed, 
$hh^{\gamma, W}(\calm,u)$ is defined to be the image of $\hat{u}^W \in Hom_{\widehat{\calw}_X}(\calm^W,\gamma_*(\calm^W))$ 
under the morphism induced on hypercohomologies by 
the following composite of morphisms:
\[
\begin{split}
R\calh om_{\widehat{\calw}_X}(\calm^W, \gamma_*(\calm^W))&\stackrel{\sim}{\leftarrow}D'_{\widehat{\calw}_X}(\calm^W)\stackrel{L}{\otimes}_{\widehat{\calw}_X}\gamma_*(\calm^W)\\
&\cong \calc^{W} _{X^a}\stackrel{L}{\otimes}_{\widehat{\calw}_{X\times X^a}}\big(\gamma_*(\calm^W)\stackrel{L}{\underline{\boxtimes}}D'_{\widehat{\calw}}(\calm^W)\big)\\
&\rightarrow \calc^{W} _{X^a}\stackrel{L}{\otimes}_{\widehat{\calw}_{X\times X^a}}\calc_{X}^{\gamma,W}=\calh\calh(\widehat{\calw}_X, \widehat{\calw}_X^\gamma).
\end{split}
\]
One can similarly provide definitions of  $\mu eu^{\gamma,W}(\calm,u)$, ${\rm eu}^W_{Q_X}(\calm)$, and ${\rm ch}^W_{Q_X}(\calm)$ that are completely analogous to the corresponding definitions in Section \ref{sec:lefschetz}. 

Recall that the support of $\widehat{\calm}:=\widehat{\cale}\otimes_{\pi^{-1}_M\cald_M}\pi^{-1}_M\calm$ in $X$ is called the characteristic variety of $\calm$ and denoted by $\text{char}(\calm)$. The following Lemma is a direct generalization of \cite[Lemma 6.5.1]{KS} to the $\gamma$ twisted setting. Let $\iota:X^{\gamma} 
\rightarrow X$ be as in Section~\ref{sec:lefschetz}.
 
\begin{proposition}\label{prop:localization} There is a natural trace density isomorphism 
\[
\calh\calh(\widehat{\calw}_X, \widehat{\calw}_X^\gamma)\xrightarrow{\mu^{\widehat{\calw}}}\iota_!\complex_{X^\gamma}((\hbar))[\dim(X^\gamma)]
\]
in the derived category of sheaves of $\complex((\hbar))$-vector spaces on $X$ such that the diagram following commutes:
\[
\begin{diagram}
\node{\calh\calh(\widehat{\cale}_X, \widehat{\cale}_X^\gamma)}\arrow{s,l,t}{}\arrow{e,t,t}{\mu^{\widehat{\cale}}}\node{\iota_!\complex_{X^\gamma}[\dim(X^\gamma)]}\arrow{s,l,t}{}\\
\node{\calh\calh(\widehat{\calw}_X, \widehat{\calw}_X^\gamma)}\arrow{e,t,t}{\mu^{\widehat{\calw}}}\node{\iota_!\complex((\hbar))_{X^\gamma}[\dim(X^\gamma)]}
\end{diagram}.
\]
Therefore, using the natural map from $H^{\dim(X^\gamma)}_{\text{char}(\calm)^\gamma}(X; \complex_{X^\gamma})$ to 
$H^{\dim(X^\gamma)}_{\text{char}(\calm)^\gamma}(X; \complex_{X^\gamma}((\hbar)))$ to 
identify $H^{\dim(X^\gamma)}_{\text{char}(\calm)^\gamma}(X; \complex_{X^\gamma})$ with its image in 
$H^{\dim(X^\gamma)}_{\text{char}(\calm)^\gamma}(X; \complex_{X^\gamma}((\hbar)))$, 
one obtains the following identities for a good $\Gamma$-equivariant  coherent $\cald_M$ module $\calm$.
\[
\begin{split}
hh^{\gamma}(\calm,\gamma)=hh^{\gamma,W}(\calm, \gamma),\  &
{\rm eu}_{Q_X}(\calm)={\rm eu}^W_{Q_X}(\calm),\\
\mu eu^{\gamma}(\calm, \gamma)=\mu eu^{\gamma, W}(\calm, \gamma),\  &{\rm ch}_{Q_X}(\calm)={\rm ch}^W_{Q_X}(\calm).
\end{split}
\]
\end{proposition}

\subsection{Orbifold Riemann-Roch theorem}
In this subsection, we describe the geometric formula for the orbifold Euler class ${\rm Eu}_Q(\calm)$ of a good $\Gamma$-equivariant coherent $\cald_M$ module $\calm$.

We recall some geometry of the orbifold $Q_X=X/\Gamma$. Note that $X^\gamma$ may have several components with different dimensions, but each component of $X^\gamma$ is a submanifold of $X$. Consider a vector bundle  $V$ on $X^\gamma$ equipped with a $\gamma$ action on each fiber. Let $R^V$ be the curvature of a connection on $V$. 
Define ${\rm ch}_\gamma(V)$ to be 
\[
{\rm ch}_\gamma(V):=\tr\left(\gamma \exp\left(\frac{R^V}{2\pi \sqrt{-1}}\right)\right)\in H^{even}(X^\gamma;\complex).
\]
Over each component of $X^\gamma$, let $N^\gamma$ be the normal bundle of $X^\gamma$ to $X$. Observe that $\gamma$ acts on fibers of $N^\gamma$ and $\wedge^\bullet N^\gamma$. Define
\[
{\rm eu}_\gamma(N^\gamma):=\sum_\bullet (-1)^\bullet {\rm ch}_\gamma(\wedge^\bullet N^\gamma)=\det\left(1-\gamma^{-1}\exp\left(\frac{-R^{\perp}}{2\pi \sqrt{-1}}\right)\right),
\]
where $R^\perp$ is the curvature of a connection on $N^\gamma$. The group $\Gamma$ naturally acts on $\sqcup_\gamma X^\gamma$: $\gamma\in \Gamma$ maps $x\in X^\alpha$ to $\gamma(x)\in X^{\gamma\alpha\gamma^{-1}}$. It is straightforward to see that 
\[
{\rm eu}_Q(N):=\sum_{\gamma\in \Gamma}{\rm eu}_{\gamma}(N^\gamma)\in \bigoplus_{\gamma\in\Gamma} H^{even}_{\text{char}(\calm)^\gamma}(X^\gamma; \complex)
\]
is invariant under the $\Gamma$ action on the above direct sum of the cohomology groups. 

Given a good $\Gamma$-equivariant coherent $\cald_M$-module $\calm$, we consider 
\[
\widehat{\calm}:=\widehat{\cale}_X\otimes_{\pi^{-1}_M\cald_M}\pi^{-1}_M\calm,
\]
equipped with a natural filtration, whose support is denoted by $\text{char}(\calm)$. The sheaf $\widehat{\cale}_X$ has a natural filtration $\{\calf^\bullet\widehat{\cale}_X\}$ by the order $\bullet$ of an operator. Let ${\rm Gr}\widehat{\cale}_X$ (resp., ${\rm Gr}\widehat{\calm}$) denote the associated graded algebra and module of $\widehat{\cale}_X$ (resp., $\widehat{\calm}$). Define 
\[
\widetilde{\rm Gr}(\widehat{\calm}):=\calo_X\otimes_{{\rm Gr}\widehat{\cale}_X} {\rm Gr}\widehat{\calm} .
\]

The symbol $\sigma_{\text{char}(\calm)}(\calm)$ is the element in the $\Gamma$-equivariant K-theory $K^{top,\Gamma}_{\text{char}(\calm)}(T^*M)$ defined by $\widetilde{\rm Gr}(\widehat{\calm})$. Restricted to $X^\gamma$, $\sigma_{\text{char}(\calm)}(\calm|_{X^\gamma})$ defines a K-theory element on $X^\gamma$ inheriting a $\gamma$-action. Applying ${\rm ch}_\gamma(-)$ to this element, one defines an element ${\rm ch}_\gamma(\sigma_{\text{char}(\calm)}(\calm|_{X^\gamma}))$ in $H^{even}_{\text{char}(\calm)^\gamma}(X^\gamma)$. As $\calm$ is $\Gamma$-equivariant, one easily checks that the collection
\[
{\rm ch}_Q(\sigma_{\text{char}(\calm)}(\calm)):=\frac{1}{|\Gamma|}\sum_{\gamma\in \Gamma} {\rm ch}_\gamma \big(\sigma_{\text{char}(\calm)} (\calm|_{X^\gamma})\big) \in \bigoplus_{\gamma\in\Gamma} H^{even}_{\text{char}(\calm)^\gamma}(X^\gamma; \complex)
\] 
is invariant under the $\Gamma$ action on the above direct sum of the cohomology groups. 

Define $\text{char}_Q(\calm)\subset IQ_X$ to be the quotient $(\sqcup_{\gamma\in \Gamma} \text{char}(\calm)^\gamma)/\Gamma$, where $\Gamma$ acts on $\sqcup_{\gamma\in\Gamma} \text{char}(\calm)^\gamma$ via its action on $\sqcup_{\gamma\in \Gamma} X^\gamma$. We are now ready to state the main theorem of this paper. 

\begin{theorem}\label{thm:obfld-ss}(Orbifold Riemann-Roch) Let $Q_M$ be the quotient of $M$ by $\Gamma$. For a good $\Gamma$-equivariant coherent $\cald_M$ module $\calm$, we have
\[
{\rm eu}_Q(\calm)=\frac{1}{m}\Big({\rm ch}_Q(\sigma_{\text{char}(\calm)}(\calm))\wedge {\rm eu}_Q(N)\wedge \pi^*Td_{IQ_M}\Big)_{\dim(I Q_X)}
\]
as a cohomology class in $H^{\dim(IQ_X)}_{\text{char}_Q(\calm)}(IQ_X;\complex)$. Here, $Td_{I Q_M}$ is the Todd class of the orbifold $IQ_M$ defined by
\[
Td_{I Q_M}:=\tr\left(\frac{\frac{R}{2\pi \sqrt{-1}}}{1-\exp^{-\frac{R}{2\pi \sqrt{-1}}}}\right),
\]
where $R$ is the curvature of a connection on the tangent bundle $TIQ_M$, $\pi: IQ_X\to IQ_M$ is the natural projection, and $m$ denotes the locally constant function on $IQ$ measuring the size of the isotropy group. 
\end{theorem}

\begin{remark}
The wedge product of differential forms on $IQ_X$ used in Theorem \ref{thm:obfld-ss} is the wedge product on each component of $IQ_X$. 
\end{remark}

The proof of Theorem \ref{thm:obfld-ss} occupies the next two subsections.  Our basic idea is to generalize the Bressler-Nest-Tsygan proof in \cite{BNT} to the $\Gamma$-equivariant setting. 

\subsection{Rees construction}
We consider the Rees ring $\calr\widehat{\cale}_X$ associated to the filtration $\{\calf^\bullet\widehat{\cale}_X\}$ of $\widehat{\cale}_X$: 
$$
\calr\widehat{\cale}_X:=\bigoplus_{p} \hbar^p \calf^p\widehat{\cale}_X. 
$$
We list a few well-known properties of $\calr\widehat{\cale}_X$ without proofs.
\begin{proposition}\label{prop:rees} (c.f. \cite{BNT})
\begin{enumerate}
\item Let ${\rm Gr}\widehat{\cale}_X$ be the associated graded ring of $\widehat{\cale}_X$ with respect to the filtration $\calf^\bullet\widehat{\cale}_X$. There are natural algebra homomorphisms 
\[
\sigma^{\calr\widehat{\cale}}: \calr\widehat{\cale}_X\stackrel{\rm Gr}{\longrightarrow} {\rm Gr}\widehat{\cale}_X\stackrel{\iota}{\longrightarrow} \calo_X.
\]
\item (\cite[I, Proposition 4.5.1]{BNT}) There is a natural flat embedding by mapping $\hbar \xi$ to $\xi$ along the fiber direction of $T^*M$,
$$
i^{\calr\widehat{\cale}}:\calr\widehat{\cale} _X\longrightarrow \widehat{\calw}_X(0). 
$$ 
\item Define $\calr\widehat{\cale}_X[\hbar^{-1}]:=\calr\widehat{\cale}_X\otimes_{\complex[[\hbar]]}\complex ((\hbar))$. There is a natural identification 
\[
\calr\widehat{\cale}_X[\hbar^{-1}]\cong \widehat{\cale}_X((\hbar)):=\widehat{\cale}_X\otimes_\complex \complex((\hbar)).
\]
\item Given a $\widehat{\cale}_X$-module $\widehat{\calm}$ with a filtration $\calf^\bullet\widehat{\calm}$, define an $\calr\widehat{\cale}_X$-module by
\[
\calr\widehat{\calm}:=\bigoplus_p \hbar^p \calf^p \widehat{\calm}. 
\]
One has the following natural isomorphisms of $\calo_X$-modules,
\begin{eqnarray}
\label{eq:hbar-1}\calr \widehat{\calm}\otimes_{\calr \widehat{\cale}_X} \calr\widehat{\cale}_X[\hbar^{-1}]&\cong& \widehat{\calm}\otimes_{\widehat{\cale}_X}\widehat{\cale}_X[\hbar^{-1},\hbar],\\
\label{eq:gr}\calr\widehat{\calm}\otimes _{\calr\widehat{\cale}_X}\calo_X&\cong&{\rm Gr}\widehat{\calm}\otimes_{{\rm Gr}\widehat{\cale}_X}\calo_X.
\end{eqnarray}
\end{enumerate}
In addition, when a finite group $\Gamma$ acts on $M$, the same results hold for 
$\widehat{\cale}_X\rtimes \Gamma$, $\widehat{\calw}_X(0)\rtimes \Gamma$, $\calr\widehat{\cale}_X\rtimes \Gamma$, ${\rm Gr}\widehat{\cale}_X\rtimes \Gamma$, $\calo_X \rtimes \Gamma$ and 
a $\Gamma$-equivariant $\widehat{\cale}_X$-module $\widehat{\calm}$. 
\end{proposition}

By Proposition \ref{prop:rees} (3), we have natural morphisms
\begin{equation}\label{eq:hbar-incl}
\iota^{\hbar^{-1}, \calr\widehat{\cale}}:\calr\widehat{\cale}_X\longrightarrow \widehat{\cale}_X[\hbar^{-1}, \hbar]\longleftarrow \widehat{\cale}_X:\iota^{
\hbar^{-1},\widehat{\cale}}.
\end{equation}

Following the idea developed in Sec. \ref{subsec:class}, when a finite group $\Gamma$ group acts on a manifold $M$ and therefore also $X=T^*M$, we view $\widehat{\cale}_{Q_X}:=\widehat{\cale}_X\rtimes \Gamma$, $\calr\widehat{\cale}_{Q_X}:=\calr\widehat{\cale}_X\rtimes \Gamma$, ${\rm Gr}\widehat{\cale}_{Q_X}:={\rm Gr}\widehat{\cale}_X\rtimes \Gamma$, $\calo_{Q_X}:=\calo_X\rtimes \Gamma$, $\widehat{\calw}(0)_{Q_X}:=\widehat{\calw}(0)_X\rtimes \Gamma$, and $\widehat{\calw}_{Q_X}:=\widehat{\calw}_X\rtimes \Gamma$ as sheaves of algebras over the orbifold $Q_X=X/\Gamma$.  And similarly a $\Gamma$-equivariant $\widehat{\cale}_X$-module $\widehat{\calm}$ is viewed as a sheaf of $\widehat{\cale}_{Q_X}:=\widehat{\cale}_X\rtimes \Gamma$-module $\widehat{\calm}_{Q_X}$ over $Q_X$. 

A crucial observation is that Theorem \ref{thm:bnt} applies to the sheaves of algebras introduced above and defines Chern character maps on the corresponding K-groups of perfect complexes of modules. These Chern characters are denoted by ${\rm ch}^A(-)$ with $A$ being relevant sheaves of algebras. We apply Proposition \ref{prop:rees} to study the Chern character of $\widehat{\calm}_{Q_X}$: 

\begin{proposition}\label{prop:gr} 
\[
\sigma_*^{\calr\widehat{\cale}}{\rm ch}^{\calr\widehat{\cale}}(\calr\widehat{\calm}_{Q_X})=\iota_*{\rm ch}^{{\rm Gr}\widehat{\cale}}({\rm Gr}\widehat{\calm}_{Q_X})={\rm ch}^\calo(\widehat{\calm}_{Q_X}\otimes_{\widehat{\cale}_{Q_X}}\calo_{Q_X})
\]
\end{proposition}
\begin{proof}
This is a direct corollary of Proposition \ref{prop:rees} (1),  and Eq. (\ref{eq:gr}).
\end{proof}

\begin{proposition}\label{prop:hbar}
\[
\iota^{\hbar^{-1}, \calr\widehat{\cale}}_*{\rm ch}^{\calr\widehat{\cale}}(\calr\widehat{\calm}_{Q_X})={\rm ch}^{\calr\widehat{\cale}[\hbar^{-1}]}(\calr\widehat{\calm}_{Q_X}\otimes_{\calr\widehat{\cale}_{Q_X}} \calr\widehat{\cale}_{Q_X}[\hbar^{-1}] )=\iota^{\hbar^{-1}, \widehat{\cale}}_*({\rm ch}^{\widehat{\cale}}(\widehat{\calm}_{Q_X})).
\]
\end{proposition}
\begin{proof} 
This is a direct corollary of Eq. (\ref{eq:hbar-1}), (\ref{eq:hbar-incl}), and Proposition \ref{prop:rees}, (3).
\end{proof}

There is a natural map $\sigma^{\widehat{\calw}(0)}: \widehat{\calw}_X(0)\to \calo_X$ defined by taking the quotient by the two sided (sheaf of)  ideal(s) generated by $\hbar$. It is easy to check that 
\begin{equation}\label{eq:sigma}
\sigma^{\calr\widehat{\cale}}=\sigma^{\widehat{\calw}(0)}\circ i^{\calr\widehat{\cale}}: \calr\widehat{\cale}_{Q_X}\longrightarrow \calo_{Q_X}. 
\end{equation}
\begin{proposition}
\label{prop:symbol}
\[
\sigma_*^{\calr\widehat{\cale}}({\rm ch}^{\calr\widehat{\cale}}(\calr\widehat{\calm}_{Q_X}))= \sigma_*^{\widehat{\calw}(0)}\circ i^{\calr\widehat{\cale}}_*({\rm ch}^{\widehat{\cale}}(\calr\widehat{\calm}_{Q_X}))={\rm ch}^{\calo}(\calr\widehat{\calm}_{Q_X}\otimes_{\widehat{\cale}_{Q_X}} \calo_{Q_X}).
\]
\begin{proof}
This is a direct corollary of Eq. (\ref{eq:sigma}).
\end{proof}
\end{proposition}

\begin{proposition}\label{prop:diagram-alg} The following diagram commutes
\[
\begin{diagram}
\node{\widehat{\calw}(0)_{Q_X}}\arrow{e,t,t}{\iota^{\hbar^{-1}, \widehat{\calw}}}\node{\widehat{\calw}_{Q_X}}\\
\node{\calr\widehat{\cale}_{Q_X}} \arrow{n,l,t}{i^{\calr \widehat{\cale} }}\arrow{e,t,t}{\iota^{\hbar^{-1}, \calr\widehat{\cale}}}
\node{\calr\widehat{\cale}_{Q_X}[\hbar^{-1}]}\arrow{n,l,t}{i^{\calr\widehat{\cale}[\hbar^{-1}]}}
\node{\widehat{\cale}_{Q_X}}\arrow{nw,t,t}{i^{\widehat{\cale}}}\arrow{w,t,t}{\iota^{\hbar^{-1}, \widehat{\cale}}}
\end{diagram},
\]
where $\iota^{\hbar^{-1}, \widehat{\calw}}$ is the natural inclusion map  $\widehat{\calw}(0)_{Q_X}\hookrightarrow \widehat{\calw}_{Q_X}$, and $i^{\calr\widehat{\cale}[\hbar^{-1}]}$ is the natural extension of $i^{\calr\widehat{\cale}}$.
\end{proposition}
\begin{proof}
This is a straightforward verification using the definitions. 
\end{proof} 

We denote the hypercohomology $H^{-\bullet}(IQ_X; \calh\calc^{per}(\widehat{\calw}_{Q_X}))$ by ${HC}_{\bullet}^{per}(\widehat{\calw}_{Q_X})$, where $\calh\calc^{per}(-)$ is the sheaf of periodic cyclic homology.  Similar notation is used for other sheaves of algebras on $Q_X$ (and for other versions of cyclic homology). Note that for any sheaf $\cala$ of algebras on $Q_X$ (or on any topological space for that matter), there is a natural map $\calh\calc^{-}(\cala) \rightarrow 
\calh\calc^{per}(\cala)$ in the derived category of sheaves of $\complex$-vector spaces on $Q_X$. Hence, one may view $\mathrm{ch}^{\cala}_{Z,i}$ (see Theorem~\ref{thm:bnt}) as a map to $HC^{per}_0(\cala)$.  The following proposition is a direct corollary of Proposition \ref{prop:diagram-alg}.
\begin{proposition}\label{prop:diagram-ch}
The following diagram commutes.
{\tiny
\[
\begin{diagram}
\node{}\node{K^0_{\Lambda}(\widehat{\calw}(0)_{Q_X})}\arrow[2]{e,t,t}{\iota^{\hbar^{-1}, \widehat{\calw}}_*}\arrow{sw,t,t}{{\rm ch}^{\widehat{\calw}(0)}}\node{}\node{K^0_{\Lambda}(\widehat{\calw}_{Q_X})}\arrow{sw,t,t}{{\rm ch}^{\widehat{\calw}}}\\
\node{HC^{per}_0(\widehat{\calw}(0)_{Q_X})}\arrow[2]{e,t,t,3}{\iota^{\hbar^{-1}, \widehat{\calw}}_*}\node{}\node{HC^{per}_0(\widehat{\calw}_{Q_X})}\\
\node{}\node{K^0_{\Lambda}(\calr\widehat{\cale}_{Q_X})}\arrow{sw,t,..}{{\rm ch}^{\calr\widehat{\cale}}}\arrow[2]{n,l,..,1}{i^{\calr\widehat{\cale}}_*}\arrow[2]{e,t,..,1}{i^{\hbar^{-1}, \calr\widehat{\cale}}_*}\node{}\node{K^0_{\Lambda}(\calr\widehat{\cale}_{Q_X}[\hbar^{-1}])}\arrow{sw,l,..,1}{{\rm ch}^{\calr\widehat{\cale}[\hbar^{-1}]}}\arrow[2]{n,t,..,1}{i^{\calr\widehat{\cale}[\hbar^{-1}]}_*}
\node{K^0_\Lambda(\widehat{\cale}_{Q_X})}\arrow{w,t,..}{\iota^{\hbar^{-1}, \widehat{\cale}}_*}\arrow{nnw,t,t}{i^{\widehat{\cale}}_*}\arrow{sw,t,t}{{\rm ch}^{\widehat{\cale}}}\\
\node{HC^{per}_0(\calr\widehat{\cale}_{Q_X})}\arrow[2]{n,t,l}{i^{\calr\widehat{\cale}}_*}\arrow[2]{e,t,t}{\iota^{\hbar^{-1},\calr\widehat{\cale}}_*}
\node{}\node{HC^{per}_0(\calr\widehat{\cale}_{Q_X}[\hbar^{-1}])}\arrow[2]{n,l,t,1}{i_*^{\calr\widehat{\cale}[\hbar^{-1}]}}\node{HC^{per}_0(\widehat{\cale}_{Q_X})}\arrow{nnw,t,t,3}{i^{\widehat{\cale}}_*}\arrow{w,t,t}{i^{\hbar^{-1}, \widehat{\cale}}_*}
\end{diagram}
\]
}
\end{proposition}
\subsection{Proof of Theorem \ref{thm:obfld-ss}}
The following  is a reformulation of  \cite[Theorem 5.13]{PPT2}.
\begin{theorem}(\cite[Theorem 5.13]{PPT2})\label{thm:alg-rr} Let $u$ be the parameter in the definition of cyclic homology, and $Q_M$ (resp., $Q_X$) be the quotient of $M$ (resp., $X$) by $\Gamma$. The following diagram commutes:
\[
\begin{diagram}
\node{HC^{per}_0(\widehat{\calw}_{Q_X}(0))}\arrow{s,t,t}{\iota^{\hbar^{-1}, \widehat{\calw}}_*}\arrow[2]{e,t,t}{\sigma^{\widehat{\calw}(0)}_*}\node{}\node{H^{-\bullet}(\Omega^\bullet_{IQ_X}((u)),d)}\arrow{s,r,t}{\wedge \frac{1}{m}{\rm eu_{Q}(N)}\wedge\pi^{-1}Td_{IQ_M}}\\
\node{HC^{per}_0(\widehat{\calw}_{Q_X})}\arrow[2]{e,t,t}{\mu^{\widehat{\calw}}}\node{}\node{H^{-\bullet}(\Omega^\bullet_{IQ_X}((\hbar))((u)),d)}
\end{diagram}
\]
\end{theorem}
\begin{proof}
Let $i^{IQ}:IQ_X\to Q_X$ be the natural forgetful map. The key observation is that the quasi-isomorphisms $\sigma^{\widehat{\calw}(0)}_*$ and
 $\mu^{\widehat{\calw}}$ constructed in \cite[Theorem 5.13]{PPT2} are morphisms of 
$i^{IQ_X}_*\underline{\complex}_{IQ_X}((u))$-modules on $Q_X$ (where $\underline{\complex}_{IQ_X}((u))$ denotes the 
(locally) constant sheaf whose space of sections over any connected open subset of $IQ_X$ is $\complex((u))$). 

For an element $x$ of ${H}^{-\bullet}(IQ_X; \calh\calc^{per}(\widehat{\calw}_{Q_X}(0)))$, 
$\sigma_*^{\widehat{\calw}(0)}(x)$ is an element of ${H}^{-\bullet}(IQ_X; \underline{\complex}((u)))$. 
Let $\mathbf{1}$ denote the trivial $\widehat{\calw}_{Q_X}(0)$-module. 
Then, we notice that $\sigma_*^{\widehat{\calw}(0)}$ maps $\sigma_*^{\widehat{\calw}(0)}(x)\cup {\rm ch}^{\widehat{\calw}(0)}(\mathbf{1})$, as an element in $HC^{per}_0(\widehat{\calw}_{Q_X}(0))$,  also to $\sigma_*^{\widehat{\calw}(0)}(x)$ in $ {H}^{-\bullet}(IQ_X; \underline{\complex}((u)))$, 
where $\cup$ is the cup product 
$$\cup\,:\,{H}^{-\bullet}(IQ_X; \underline{\complex}((u))) \otimes {H}^{-\bullet}(Q_X; \calh\calc^{per}(\widehat{\calw}_{Q_X}(0)))
 \rightarrow {H}^{-\bullet}(Q_X; \calh\calc^{per}(\widehat{\calw}_{Q_X}(0)))\,\text{.} $$

 As $\sigma_*^{\widehat{\calw}(0)}$ is a quasi-isomorphism, we conclude that $x=[\sigma_*^{\widehat{\calw}(0)}(x)]\cup [{\rm ch}^{\widehat{\calw}(0)}(\mathbf{1})]$ is in \\
${H}^{-\bullet}(Q_X; \calh\calc^{per}(\widehat{\calw}_{Q_X}(0)))$. Since $\mu^{\widehat{\calw}}$ is a $H^{-\bullet}(IQ_X, \underline{\complex}((u)))$-module map, we have 
$$\mu^{\widehat{\calw}}(\iota^{\hbar^{-1},\widehat{\calw}}_*(x)) = \sigma_*^{\widehat{\calw}(0)}(x) \cup \mu^{\widehat{\calw}}(\iota^{\hbar^{-1, \widehat{\calw}}}_*({\rm ch}^{\widehat{\calw}(0)}(\mathbf{1})) )
\,\text{.}  $$
This reduces the proof to computing $\mu^{\widehat{\calw}}(\iota^{\hbar^{-1, \widehat{\calw}}}_* ({\rm ch}^{\widehat{\calw}(0)}(\mathbf{1}))) $. 
This computation is done by the same proof as that of \cite[Theorem 5.13]{PPT2}, but in the holomorphic setting. 
As is computed in \cite[4.5.1]{BNT}, the characteristic class of the quantization $\widehat{\calw}_X(0)$ is equal to $\frac{1}{\hbar}\omega+\frac{1}{2} \pi_M^*c_1(TX)$ with $\omega$ the symplectic form on $T^*M$. 
Substituting this characteristic class into~\cite[Theorem 5.13]{PPT2} yields the desired identity. 
\end{proof}
\noindent{\em Proof of Theorem \ref{thm:obfld-ss}}: The Euler class is the top degree component of the 
Chern character, which is computed by the following steps. 
\begin{eqnarray*}
&{\rm ch}_Q(\calm)&=\mu^{\widehat{\cale}}\circ {\rm ch}^{\widehat{\cale}}(\widehat{\calm})\qquad \text{Definition \ref{dfn:class}}\\
&&=\mu^{\widehat{\calw}}\circ i_*^{\widehat{\cale}}{\rm ch}^{\widehat{\cale}}(\calm)\qquad \text{Proposition \ref{prop:localization}}\\
&&=\mu^{\widehat{\calw}}\circ i_*^{\calr\widehat{\cale}[\hbar^{-1}]} \circ i_*^{\hbar^{-1},\widehat{\cale}}{\rm ch}^{\widehat{\cale}}(\calm)\qquad \text{right front triangle of Proposition \ref{prop:diagram-ch}}\\
&&=\mu^{\widehat{\calw}}\circ i_*^{\calr\widehat{\cale}[\hbar^{-1}]} \circ\iota^{\hbar^{-1}, \calr\widehat{\cale}}_* {\rm ch}^{\calr\widehat{\cale}}(\calr\widehat{\calm}_{Q_X})\qquad \text{Proposition \ref{prop:hbar}}\\
&&=\mu^{\widehat{\calw}}\circ \iota_*^{\hbar^{-1},\widehat{\calw}}\circ i^{\calr\widehat{\cale}}_* \circ  {\rm ch}^{\calr\widehat{\cale}}(\calr\widehat{\calm}_{Q_X})\qquad \text{right front square of Proposition \ref{prop:diagram-ch}}\\
&&=\frac{1}{m}\sigma_*^{\widehat{\calw}(0)}\circ i_*^{\calr\widehat{\cale}}\circ {\rm ch}^{\calr\widehat{\cale}}(\calr\widehat{\calm}_{Q_X})\wedge {\rm eu}_Q(N)\wedge \pi^{-1}Td_{IQ_M}\qquad \text{Theorem \ref{thm:alg-rr}}\\
&&=\frac{1}{m}{\rm ch}^{\calo}(\calr\widehat{\calm}_{Q_X}\otimes_{\widehat{\cale}_{Q_X}} \calo_{Q_X})\wedge {\rm eu}_Q(N)\wedge \pi^{-1}Td_{IQ_M}\qquad \text{Proposition \ref{prop:symbol}}\\
&&=\frac{1}{m}{\rm ch}^{\calo}({\rm Gr}\widehat{\calm}_{Q_X}\otimes_{{\rm Gr}\widehat{\cale}_{Q_X}}\calo_{Q_X}) \wedge {\rm eu}_Q(N)\wedge \pi^{-1}Td_{IQ_M}\qquad \text{Eq. (\ref{eq:gr})}\\
&&=\frac{1}{m}{\rm ch}_Q(\sigma_{\text{char}(\calm)}(\calm))\wedge {\rm eu}_Q(N)\wedge \pi^{-1}Td_{IQ_M}.\qquad\qquad \square
\end{eqnarray*}

\end{document}